\documentclass[10pt]{article}
\usepackage[latin1]{inputenc}
\usepackage{amsmath}
\usepackage{amsfonts}
\usepackage{amssymb}
\usepackage{amsthm}

\usepackage{hyperref}

\usepackage[all]{xy}
\usepackage[usenames,dvipsnames]{color}
\usepackage[dvips]{graphicx}

\newtheorem{maintheorem}{Theorem}

\newtheorem{teo}{Theorem}[section]
\newtheorem{df}[teo]{Definition}
\newtheorem{pro}[teo]{Proposition}
\newtheorem{cor}[teo]{Corollary}
\newtheorem{lem}[teo]{Lemma}

\newtheorem{rem}[teo]{Remark}
\newtheorem{ex}[teo]{Examples}

\newtheorem{rems}[teo]{Remarks}

\newtheorem{question}[teo]{Question}
\newtheorem{conjecture}[teo]{Conjecture}

\long\def\symbolfootnote[#1]#2{\begingroup
\def\thefootnote{\fnsymbol{footnote}}\footnote[#1]{#2}\endgroup} %Para usar footnotes con simbolos mediante el comando symbolfootnote

\title{Analytic continuation of holonomy germs of Riccati foliations along Brownian paths.}
\author{Nicolas Hussenot Desenonges}

\begin{document}
\maketitle

\begin{abstract}
Consider a Riccati foliation whose monodromy representation is non elementary and parabolic and consider a non invariant section of the fibration whose associated developing map is onto. We prove that any holonomy germ from any non invariant fibre to the section can be analytically continued along a generic Brownian path. To prove this theorem, we prove a dual result about complex projective structures: let $\Sigma$ be a hyperbolic Riemann surface of finite type endowed with a branched complex projective structure. Such a structure gives rise to a non-constant holomorphic map $\mathcal{D}:\tilde{\Sigma}\rightarrow \mathbb{C}\mathbb{P}^1$ from the universal cover of $\Sigma$ to the Riemann sphere $\mathbb{C}\mathbb{P}^1$ which is $\rho$-equivariant for a morphism  $\rho:\pi_1(\Sigma)\rightarrow PSL(2,\mathbb{C})$. The dual result is the following: if the monodromy representation $\rho$ is parabolic and non elementary and if $\mathcal{D}$ is onto, then for almost every Brownian path $\omega$ in $\tilde{\Sigma}$, $\mathcal{D}(\omega(t))$ does not have limit when $t$ goes to $\infty$. If moreover the projective structure is of parabolic type, we also prove that although $\mathcal{D}(\omega(t))$ does not converge, it converges in the Ces\`{a}ro sense.
\end{abstract}
%\newpage
%\tableofcontents
%\newpage

%\input{intro.tex}

%\input{chap1.tex}

%\input{chap2.tex}

\section{Introduction.}

Given a complex algebraic foliation, the study of the holonomy maps is crucial since they encode the dynamics of the leaves. This paper is devoted to the problem of analytic continuation of these holonomy maps. This problem which goes back to the times of Painlev\'{e} regained interest recently with the works of F. Loray \cite{L}, Y. Il'yashenko \cite{Il} and G. Calsamiglia, B. Deroin, S. Frankel, A. Guillot \cite{CDFG}. 

Let us explain the context. Consider the following diferential equation in $\mathbb{C}^2$:
\begin{equation}\label{eq-dif}
\frac{dy}{dx}=\frac{P(x,y)}{Q(x,y)}
\end{equation}
where $P$ and $Q$ are polynomials in $\mathbb{C}[X,Y]$ without common factors.

The solutions of \eqref{eq-dif} define a singular holomorphic foliation of complex dimension $1$ in $\mathbb{C}^2$ which can be extended to a singular holomorphic foliation $\mathcal{F}$ of $\mathbb{C}\mathbb{P}^2$.

Let $C_0$, $C_1$ be two complex curves in $\mathbb{C}\mathbb{P}^2$ and $L$ be a leaf of $\mathcal{F}$ which intersects $C_0$ in $p_0$ and $C_1$ in $p_1$. Assume that $p_0$ and $p_1$ are not singularities of the foliation and that the curve $C_0$ (resp. $C_1$) is transverse to $\mathcal{F}$ in $p_0$ (resp. $p_1$). Let $\gamma:[0,1]\rightarrow L$ be a continuous path such that $\gamma(0)=p_0$ and $\gamma(1)=p_1$. Then, one can find a continuous family $\gamma_p:[0,1] \rightarrow\mathbb{C}\mathbb{P}^2$ of paths parameterized by $p \in C_0$ close enough to $p_0$ such that:
\begin{enumerate}
\item $\gamma_p(0)=p$.
\item $\gamma_p(1) \in C_1$.
\item $\gamma_{p_0}=\gamma$.
\item For all $p$, $\gamma_p$ belongs to the leaf through $p$.
\end{enumerate}
The germ of the holomorphic map $p\mapsto \gamma_p(1)$ in $p_0$ is uniquely determined by the relative (i.e. with fixed endpoints) homotopy class of $\gamma$ under the above conditions and is called the \emph{holonomy germ} associated to $\gamma$. 

A rather general question is to define the domain of definition of such a germ:

In \cite{L}, F.Loray conjectures the following:

\begin{conjecture}[Loray]
Let $\mathcal{F}$ be a singular holomorphic foliation in $\mathbb{C}\mathbb{P}^2$. Let $L_1$ and $L_2$ be two non invariant projective lines and $h:(L_1,p_1)\rightarrow(L_2,p_2)$ be a holonomy germ. Then $h$ can be analytically continued along any continuous path which avoids a countable set of points called singularities of $h$.
\end{conjecture}

This was motivated by the following result which can be found in \cite[Theorem 1.1]{CDFG} and which is a consequence of theorem $1$ of Painlev\'{e} (see \cite{L}): if the polynomials $P$ and $Q$ of equation \eqref{eq-dif} are such that $w=Pdx-Qdy$ is a closed one-form, then Loray's conjecture is true.

In the same vein, Y. Il'yashenko asks the following \cite{Il}:
\begin{question}[Ilyashenko]
Consider the foliation in $\mathbb{C}^2$ associated to equation \eqref{eq-dif} and let $h:(L_1,p_1)\rightarrow(L_2,p_2)$ be a holonomy germ between two lines.
Can $h$ be analytically continued along a generic ray emerging from $p_1$?
\end{question}

In \cite{CDFG}, the authors prove that Loray's conjecture fails to be true. More precisely, they prove the followings:
\begin{itemize}
\item  For an algebraic foliation of $\mathbb{C}\mathbb{P}^2$ with hyperbolic singularities and without invariant curves (these are generic properties), there is a holonomy germ between a projective line and a curve whose set of singularities contains a Cantor set.
\item  There exists algebraic foliations of $\mathbb{C}\mathbb{P}^2$ admitting a holonomy germ $h:(L_1,p_1)\rightarrow(L_2,p_2)$ between complex lines whose set of singularities is the whole $L_1$.
\end{itemize}
Our main result is more particularly linked to the second assertion. To see this, let us explain briefly how they build such a foliation. They consider a parabolic projective structure on the punctured Riemann sphere whose monodromy group is dense in $PSL(2,\mathbb{C})$. Suspending the monodromy representation, one obtains a $\mathbb{C}\mathbb{P}^1$-fibre bundle over the punctured Riemann sphere endowed with a non singular foliation transverse to every fibre and a section $\Delta$ (given by the developing map). There exist local models (introduced by M. Brunella in \cite{B}) over the cusps which allow to compactifie the $\mathbb{C}\mathbb{P}^1$-fibre bundle, the foliation and the section . After the compactification, one gets a singular holomophic foliation on a $\mathbb{C}\mathbb{P}^1$-fibre bundle over $\mathbb{C}\mathbb{P}^1$ whose fibers are transverse to the foliation, except the ones over the punctures which are invariant lines containing the singularities of the foliation. Now, consider a holonomy germ $h$ between a transverse fibre and the section given by a developing map of the projective structure. Then $h$ is a local inverse of the developing map. If the monodromy group of the projective structure is dense in $PSL(2,\mathbb{C})$, they prove that $h$ has full singular set. The $\mathbb{C}\mathbb{P}^1$-bundles over $\mathbb{C}\mathbb{P}^1$ are parameterized by an integer $n\geq0$. Choosing conveniently the local models around the cusps, this number is $n=1$, so that the ambient space is the first Hirzebruch surface $\mathbb{F}_1$ which has a unique exceptional curve disjoint from $\overline{\Delta}$. Blow down gives $\mathbb{C}\mathbb{P}^2$ with the desired property.

This paper is based on the following observation: with the same hypothesis, even if the germ $h$ has full singular set, $h$ can be analytically continued along a generic Brownian path, i.e. the Brownian motion does not see this full set of singularities.

The foliations previously defined on Hirzebruch surfaces are examples of Riccati foliations. More generally, a \emph{Riccati foliation} is the data of $(\Pi,M,X,\mathcal{F})$ where $M$ is a compact complex surface, $X$ is a compact Riemann surface, $\Pi:M\to X$ is a $\mathbb{C}\mathbb{P}^1$-fibre bundle and $\mathcal{F}$ is a singular holomorphic foliation transverse to all the fibers except a finite number of them which are invariant lines for the foliation and contain the singularities. The main theorem of this paper is:

\begin{maintheorem}\label{theoRiccati}
Let $\mathcal{F}$ be a Riccati foliation with a parabolic and non elementary monodromy group. Let $F$ be a non invariant fiber and $s_0$, $s_1$ be two sections of the bundle. Denote by $\overline{S_0}$ and $\overline{S_1}$ the images of $X$ by $s_0$ and $s_1$. Endow $F$, $\overline{S_0}$ and $\overline{S_1}$ with complete metrics in their conformal class. Assume moreover that the developing map associated to $\overline{S_0}$ is onto.
\begin{enumerate}
\item If $h:(F,p)\to (\overline{S_0},p_0)$ is a holonomy germ, then $h$ can be analytically continued along almost every Brownian path in $F$ starting at $p$.
\item If $h:(\overline{S_1},p_1)\to(\overline{S_0},p_0)$ is a holonomy germ, then $h$ can be analytically continued along almost every Brownian path in $\overline{S_1}$ starting at $p_1$.
\end{enumerate}
% (the brownian motion is the one associated to any complete metric in $L_1$ in its conformal class)
\end{maintheorem}

\begin{rem}\label{rem-section} A holomorphic $\mathbb{C}\mathbb{P}^1$-fibre bundle always admits a holomorphic section (see \cite[p 139]{BPV}). \end{rem}

Theorem \ref{theoRiccati} will be a consequence of a theorem concerning complex projective structures that we explain now.

\paragraph{Complex projective structures. }
Let $\Sigma$ be a Riemann surface. A \emph{branched complex projective structure} in $\Sigma$ is a $(PSL(2,\mathbb{C}),\mathbb{C}\mathbb{P}^1)$-structure where $\mathbb{C}\mathbb{P}^1$ is the Riemann sphere and $PSL(2,\mathbb{C})$ is the group of M\"{o}bius transformations acting on $\mathbb{C}\mathbb{P}^1$. Such a structure gives rise to a non-constant holomorphic map $\mathcal{D}:\tilde{\Sigma}\rightarrow \mathbb{C}\mathbb{P}^1$ from the universal cover of $\Sigma$ to the Riemann sphere $\mathbb{C}\mathbb{P}^1$  and to a morphism  $\rho:\pi_1(\Sigma)\rightarrow PSL(2,\mathbb{C})$ satisfying the following equivariance relation:
$$\forall x \in \tilde{\Sigma},\, \forall \alpha \in \pi_1(\Sigma),\,  \mathcal{D}(\alpha.x)=\rho(\alpha).\mathcal{D}(x)$$
The map $\mathcal{D}$ (well defined up to a post-composition by a M\"{o}bius transformation) is called developing map, the morphism $\rho$ is called monodromy representation, and the group $\rho(\pi_1(\Sigma))$ is called monodromy group (see section \ref{sectionsp} for more details on projectives structures).

If $\Sigma$ is not compact, we will need parabolicity hypothesis around the cusps: a representation is said to be \emph{parabolic} if the holonomy around each cusp is parabolic (i.e. it is conjugated to the group generated by the transformation $z\mapsto z+1$). A complex projective structure is said to be \emph{parabolic} if in some coordinate $z$ around each puncture, some developing map writes $\mathcal{D}(z)=\frac{1}{2i\pi}\log z$. Our theorem \ref{theoprincipal} is proved under the hypothesis of parabolicity of the monodromy representation while theorem \ref{theo2} is proved under the stronger hypothesis of parabolicity of the projective structure.

\paragraph{The image of a generic Brownian path by the developing map. }
In \cite{CDFG}, the authors prove that if the monodromy group is a dense subgroup of $PSL(2,\mathbb{C})$ and if $h$ is a germ of $\mathcal{D}^{-1}$ in $z_0$, then the set of singularities for the analytic continuation of $h$ is all the Riemann sphere $\mathbb{C}\mathbb{P}^1$ (see proposition \ref{proCDFG}). In other words, for any point $z$ in $\mathbb{C}\mathbb{P}^1$, there is a continuous path $c$ from $z_0$ to $z$ such that $h$ cannot be analytically continued along $c$ (we will give a proof of this fact in section \ref{sectionDeroin}). As it has been explained earlier, with the same hypothesis, $h$ can be analytically continued along a generic Brownian path (i.e. the Brownian motion does not see this full set of singularities). More precisely:

\begin{maintheorem}\label{theoprincipal}

Let $\Sigma$ be a Riemann surface of finite type endowed with a branched projective structure. Let $\mathcal{D}:\tilde{\Sigma}\rightarrow \mathbb{C}\mathbb{P}^1$ be a developing map, $\rho:\pi_1(\Sigma)\rightarrow PSL(2,\mathbb{C})$ be the monodromy representation associated to $\mathcal{D}$. Let $(x_0,z_0)$ be a couple of points in $\tilde{\Sigma}\times\mathbb{C}\mathbb{P}^1$ such that $\mathcal{D}(x_0)=z_0$ and let $h$ be the germ of $\mathcal{D}^{-1}$ such that $h(z_0)=x_0$. We also define the Brownian motion in $\tilde{\Sigma}$ as the one associated to the hyperbolic metric with constant curvature $-1$ and the Brownian motion in $\mathbb{C}\mathbb{P}^1$ as the one associated to any complete metric in its conformal class. We have:

Assume that $\mathcal{D}$ is onto and that the monodromy group $\Gamma=\rho(\pi_1(\Sigma))$ is parabolic and non elementary.

Then, the two following equivalent assertions are satisfied:
\begin{enumerate}
\item For almost every Brownian path $\omega$ starting from $x_0$, $\mathcal{D}(\omega(t))$ does not have any limit when $t$ goes to $\infty$.
\item For almost every Brownian path $\omega$ starting from $z_0$, the map $h$ can be analytically continued along $\omega([0,\infty[)$.
\end{enumerate}

\end{maintheorem}

The equivalence of the two assertions is a direct consequence of the conformal invariance of the Brownian motion. In order to prove the first assertion, we will use the discretization procedure of Furstensberg-Lyons-Sullivan. In our context, this procedure associates to every Brownian path $\omega$ in $\tilde{\Sigma}$ a sequence $X_n(\omega)$ of elements of $\pi_1(\Sigma)$ which corresponds more or less to the sequence of fundamental domains visited by $\omega$. The sequence $X_n(\omega)$ turns out to be the realisation of a right random walk, i.e. $X_{n+1}(\omega)=X_n(\omega)\cdot\gamma_{n+1}(\omega)$, the $\gamma_n(\omega)$ being independent and identically distributed. Pushing $X_n(\omega)$ forward by $\rho$ gives a right random walk $Y_n(\omega)$ in $\rho(\pi_1(\Sigma))<PSL(2,\mathbb{C})$. Random walks in such matricial groups have been widely studied. A classical result of the theory is the following: if the support of the measure $\mu$ defining the random walk $Y_n$ is non elementary and if $\nu$ is a $\mu$-stationnary measure on $\mathbb{C}\mathbb{P}^1$, then for almost every $\omega$, there exists $z(\omega)\in\mathbb{C}\mathbb{P}^1$ such that:
$$Y_n(\omega)\cdot\nu\underset{n\to\infty}\longrightarrow\delta_{z(\omega)}.$$
In view of this property, theorem \ref{theoprincipal} is surprising because one could think at first glance that this contraction property would implie that $\mathcal{D}(\omega(t))\underset{t\to\infty}\to z(\omega)$.

In section \ref{sectionproof1}, we will give a new statement of the last theorem including the case where $\mathcal{D}$ is not onto. In this case, the opposite conclusion holds: for almost every Brownian path $\omega$ starting from $x_0$, there is a point $z(\omega)$ such that $\lim\limits_{t \to\infty}\mathcal{D}(\omega(t))=z(\omega)$, which is equivalent to the following: for almost every Brownian path $\omega$ starting from $z_0$, the map $h$ cannot be analytically continued along $\omega([0,\infty[)$.

\paragraph{The family of harmonic measures. }
At the beginning of this study, we did not think that theorem \ref{theoprincipal} was realistic. On the contrary, we expected to prove that, in both cases ($\mathcal{D}$ onto and $\mathcal{D}$ not onto), the following holds: for all $x$ in $\tilde{\Sigma}$, for almost every Brownian path $\omega$ starting at $x$, there is a point $z(\omega)$ such that $\lim\limits_{t \to\infty}\mathcal{D}(\omega(t))=z(\omega)$. The existence of such a point $z(\omega)$ would allow to associate to any projective structure on $\Sigma$ a family of measures $(\nu_x)_{x\in\tilde{\Sigma}}$ in $\mathbb{P}^1$ in the following way: if $\mathbb{P}_{x}$ is the Wiener-measure on the set $\Omega_{x}$ of continuous paths starting at $x$, for any Borel set $A$ in $\mathbb{P}^1$, we should have defined:

$$\nu_x(A)=\mathbb{P}_{x}\left(w\in\Omega_{x} \text{ such that } z(\omega)\in A\right).$$

Altough $\mathcal{D}(\omega(t))$ does not converge when $t$ goes to $\infty$ (in the case where $\mathcal{D}$ is onto) in the classical sense, $\mathcal{D}(\omega(t))$ converges almost surely in the Ces\`{a}ro sense. This means:

\begin{maintheorem}\label{theo2}
Let $\Sigma$ be a hyperbolic Riemann surface of finite type endowed with a branched complex projective structure of parabolic type. Let $\mathcal{D}:\tilde{\Sigma}\rightarrow \mathbb{C}\mathbb{P}^1$ be a developing map and $\rho:\pi_1(\Sigma)\rightarrow PSL(2,\mathbb{C})$ be the representation associated with $\mathcal{D}$. Assume moreover that $\rho$ is non-elementary. 
Then, for every $x$ in the universal cover $\tilde{\Sigma}$ and for almost every Brownian path $\omega$ starting from $x$, there exists $z(\omega) \in \mathbb{C}\mathbb{P}^1$ such that:
$$\frac{1}{t}\cdot \displaystyle\int_0^t\delta_{\mathcal{D}(\omega(s))}\cdot ds\underset{t\to\infty}\longrightarrow\delta_{z(\omega)}.$$

\end{maintheorem}

%\begin{rem}
%In theorem (\ref{theoprincipal}), we made  the assomption that $\Sigma$ was of finite type and that the projective structure on $\Sigma$ was ?tait ?ventuellement branch?e. Ici, la surface $\Sigma$ est compacte et la structure projective sur celle-ci est non branch?e. Ces hypoth?ses plus fortes rendent la preuve moins techniques mais ne sont, ? mon avis, pas n?cessaire.
%\end{rem}

Then, to any complex projective structure on $\Sigma$ satisfying the hypothesis of the previous theorem, one can associate a family $(\nu_x)_{x\in\tilde{\Sigma}}$ of harmonic measures on $\mathbb{C}\mathbb{P}^1$:  it is the distribution law of the point $z(\omega)$ (given by the previous theorem) for a Brownian path starting at $x$. This family of measures gives interesting informations about the projective structure. It has been recently studied by Deroin and Dujardin in \cite{DD}. In a recent work in collaboration with S. Alvarez \cite{AH}, we prove the following: for all $x\in\tilde{\Sigma}$, the image of a generic geodesic ray starting at $x$ by the developing map has a limit in $\mathbb{C}\mathbb{P}^1$. The distribution law of this limit point (with respect to the angular measure at $x$ for the Poincare metric) is a measure $\mu_x$ which is proved to be equal to $\nu_x$.

\paragraph{Organisation of the paper. } 
First, section \ref{sectionsp} is devoted to the basic definitions and examples concerning branched projective structures. Then, section \ref{sectionDeroin} deals with generalities about analytic continuation of holomorphic maps. Section \ref{sectionrw}, where we prove a contraction property for random walks in $PSL(2,\mathbb{C})$ and section \ref{sectiondiscret} where we explain the discretization procedure of Furstenberg-Lyons-Sullivan are the necassary backgrounds for the proof of theorem \ref{theoprincipal} in section \ref{sectionproof1}. In section \ref{sectionRiccati}, we prove theorem \ref{theoRiccati} and in the last one \ref{sectionproof2}, we prove theorem \ref{theo2}.\newline

\paragraph{Aknowledgments. }
The major part of this paper comes from the second part of my Phd thesis \cite{Hu}. I am very grateful to Ga\"{e}l Meigniez, Fr\'{e}d\'{e}ric Math\'{e}us and Bertrand Deroin for their precious help during all these years.

\section{Projectives structures.}\label{sectionsp}

This part gives basic concepts about complex projective structures which will be useful in the sequel. For further insights on this notion, we refer the reader to the survey of D. Dumas \cite{Du}.

\begin{df}
Let $\Sigma$ be a Riemann surface. A \emph{branched projective structure} on $\Sigma$ is a maximal atlas $(\phi_i:U_i\rightarrow \mathbb{C}\mathbb{P}^1)$ where the $U_i$ are open sets in $\Sigma$ and the $\phi_i$ are non constant holomorphic maps on $U_i$ such that on the intersection of two domains $u_i\cap U_j$, the relation $\phi_i=\gamma_{ij}\circ\phi_j$ holds for some M\"{o}bius transformation (i.e. for some element of $PSL(2,\mathbb{C})$).
\end{df}

\begin{rem}
We use the term branched because the charts may have critical points. When $\mathcal{D}$ does not have such branching points, the structure is simply called projective structure. 
\end{rem}

Let $\phi_i:U_i\rightarrow V_i$ be a chart of such an atlas. If $U_j$ is an other chart such that $U_i\cap U_j\neq 0$, then the map $\gamma_{ij}\circ\phi_j:U_j\rightarrow \mathbb{C}\mathbb{P}^1$ is equal to $\phi_i$ on $U_i\cap U_j$ and allows to continue $\phi_i$ to $U_j$. Continuing  this way, we obtain a globally defined holomorphic map whose domain of definition is the universal covering space $\tilde{\Sigma}$. This map, denoted by $\mathcal{D}:\tilde{\Sigma}\rightarrow \mathbb{C}\mathbb{P}^1$ is called a \emph{developing map}. $\mathcal{D}$ is defined up to a post-composition by a M\"{o}bius transformation. 

Associated with this, we can define a morphism $\rho:\pi_1(\Sigma)\rightarrow PSL(2,\mathbb{C})$ called \emph{monodromy representation} which satisfies the following equivariance relation:
\begin{eqnarray*}
\forall x \in \tilde{\Sigma},\, \forall \alpha \in \pi_1(\Sigma),\,  \mathcal{D}(\alpha.x)=\rho(\alpha).\mathcal{D}(x)
\end{eqnarray*}
The group $\Gamma:=\rho(\pi_1(\Sigma))$ is called \emph{monodromy group} of the branched projective structure. As the developing map $\mathcal{D}$ is defined up to a post-composition by a M\"{o}bius transformation, $\Gamma$ is defined up to a conjugacy by this transformation.

In this paper, we will consider Riemann surfaces of finite type, i.e. compact Riemann surfaces with a finite number of points deleted. Our theorems \ref{theoprincipal} and \ref{theo2} concerning projective structures both assume that the monodromy representation is non elementary. Theorem \ref{theoprincipal} assumes that the monodromy representation is parabolic and theorem \ref{theo2} assumes that the projective structure is parabolic (which is a stronger condition). Here, we recall the definitions of these notions:

\begin{df}
\begin{itemize}
\item A representation $\rho:\pi_1(\Sigma)\to PSL(2,\mathbb{C})$ is said to be parabolic if the monodomy is parabolic around each puncture (i.e. it is conjugated to the group generated by the transformation $z\mapsto z+1$).
\item A branched projective structure on a Riemann surface of finite type is said to be parabolic if for any puncture $p$, there is a neighborhood $V_p$ of $p$ and a biholomorphism $\phi:D(0,e^{-2\pi})-\{0\}\to V_p$ such that some developing map satisfies $\mathcal{D}\circ\phi(z)=\frac{1}{2i\pi}\log z.$ (in this definition, the developing must be seen as a multivalued holomorphic map from $\Sigma$ to $PSL(2,\mathbb{C})$).
\item A subgroup $\Gamma$ of $PSL(2,\mathbb{C})$ is said to be elementary if there exists a finite set in $\mathbb{C}\mathbb{P}^1$ which is globally invariant by the action of $\Gamma$or if it is conjugate to a subgroup of the projective special unitary group $PSU(2,\mathbb{C})$. A representation $\rho:\pi_1(\Sigma)\to PSL(2,\mathbb{C})$ is said to be elementary if $\rho(\pi_1(\Sigma))$ is elementary.
\end{itemize}
\end{df}

\begin{rem}\label{rem1}
\begin{itemize}
\item If the monodromy representation is non elementary, then the Riemann surface is necessarely hyperbolic. Indeed, The monodromy group of a branched projective structure on the sphere is trivial (since $\pi_1(\mathbb{C}\mathbb{P}^1)$ is trivial) and the monodromy group of a branched projective structure on a parabolic Riemann surface $\Sigma$ is abelian (since $\pi_1(\Sigma)$ is). 
\item At the universal covering level, the parabolicity of a projective structure at a puncture $p$ implies that the developing map in any connected component of the preimage of $V_p$ is holomorphically conjugated to the inclusion map. More precisely, consider the universal covering map $q:\mathbb{H}_{\geq 1}=\{Im z\geq 1\}\to D(0,e^{-2\pi})-\{0\}$, $\tau\mapsto e^{2i\pi\tau}$. Let $\mathcal{H}_p$ be a connected component of $proj^{-1}(V_p)$ where $proj:\tilde{\Sigma}\to\Sigma$ is the universal covering map. If $\phi:D(0,e^{-2\pi})-\{0\}\to V_p$ satisfies $\mathcal{D}\circ\phi(z)=\frac{1}{2i\pi}\log z$, then lifting $\phi$ by $q$ and $proj$, one gets a biholomorphism $\tilde{\phi}:\mathbb{H}_{\geq 1}\to\mathcal{H}_p$ satisfying $\mathcal{D}\circ\tilde{\phi}(\tau)=\tau$. Moreover, it is proved in \cite{AH} that $\tilde{\phi}$ is in fact bilipschitz for the hyperbolic metrics.
\item If $\mathcal{D}(z)=\frac{1}{2i\pi}\log z$ in a coordinate $z$ around a puncture, then $\mathcal{D}(e^{2i\pi}z)=\mathcal{D}(z)+1$. So the parabolicity of the projective structure implies the parabolicity of the monodromy representation. But the converse is false in general. Indeed, on the puncture disc, the projective structure given by $\mathcal{D}_n(z)=\frac{1}{2i\pi} \log z+\frac{1}{z^n}$ has a parabolic monodromy representation ($\mathcal{D}_n(e^{2i\pi}z)=\mathcal{D}_n(z)+1$) but it is not parabolic for $n\in \mathbb{N}^*$ (to see this, one can check for example that $\mathcal{D}_n(z)$ does not have limit when $z$ goes to $0$).
\end{itemize}
\end{rem}

\begin{ex}\label{exstructproj}
\begin{enumerate}
\item Let $\Sigma$ be a hyperbolic Riemann surface. The universal covering space of $\Sigma$ is the upper half plane $\mathbb{H}$ and $\Sigma=\mathbb{H}/\Gamma$ where $\Gamma$ is a subgroup of $PSL(2,\mathbb{R})$ whose action on $\mathbb{H}$ is free and properly discontinuous.
The couple $(\mathcal{D},\rho)=(i:\mathbb{H}\hookrightarrow \mathbb{C}\mathbb{P}^1,\,i:\Gamma\hookrightarrow PSl(2,\mathbb{C}))$ (where $i$ is the inclusion map) defines a projective structure on $\Sigma$ called the \emph{uniformizing projective structure} of $\Sigma$.
\item Let $\Gamma$ be a Kleinian group (i.e. a discrete subgroup of $PSL(2,\mathbb{C})$) such that the set of discontinuity $\Omega(\Gamma)\in \mathbb{C}\mathbb{P}^1$ is not vacuous. The quotient $\Omega(\Gamma)/\Gamma$ is a Riemann surface which can be endowed with a projective structure in the following way: we cover $\Omega(\Gamma)/\Gamma$ by open sets $U_i$ small enough and we choose local inverses $s_i$  of the projection $p:\Omega(\Gamma)\rightarrow\Omega(\Gamma)/\Gamma$ defined on $U_i$. The $s_i:U_i\rightarrow \Omega(\Gamma)\subset \mathbb{C}\mathbb{P}^1$ define an atlas of $\Sigma$ whose transition functions are elements of $\Gamma$ (i.e. M\"{o}bius transformations). Note that by Ahlfors' finitness theorem \cite{Ah}, the Riemann surface $\Omega(\Gamma)/\Gamma$ is of finite type and the projective structure is parabolic.
\item In the two previous examples, the developing map is not onto. Starting with the uniformizing projective structure of $\Sigma$ as in example $1.$, there is a surgery operation introduced by Heijal \cite{He} and called \emph{grafting}, that produces new projective structures having the same monodromy representation but such that the new developing map is not onto.
\end{enumerate}
\end{ex}

\section{Analytic continuation.} \label{sectionDeroin}

Recall that one of the goals of this paper is to show that, with some good assumptions on the projective structure, any local inverse $h$ of the developing map can be analytically continued along a generic Brownian path. In this part, following the paper \cite{CDFG}, we show that, however, there are many paths along which $h$ cannot be analytically continued. Let us start with some basic definitions about analytic continuation of holomorphic maps.

Let $C_0$ and $C_1$ be two Riemann surfaces and a germ of holomorphic map $h:(C_0,p_0)\rightarrow(C_1,p_1)$. Let $\tau:[0,t]\rightarrow C_0$ be a continuous path such that $\tau(0)=p_0$. We say that $\tau$ is covered by the sequence of open discs $D_1,\cdots,D_n$ if there is a sequence of times $0=t_0<t_1<\cdots<t_n=t$ such that $\tau([t_k,t_{k+1}])\subset D_{k+1}$. We say that $h$ can be \emph{analytically continued} along $\tau([0,t])$ if there is a sequence of discs $D_1,\cdots,D_n$ covering $\tau$ and  holomorphic maps $f_k:D_k\rightarrow C_1$ such that the germ of $f_1$ in $p_0$ is $h$ and such that for all $k \in \{1,\cdots,n-1\}$, we have $f_k=f_{k+1}$ on $D_k \cap D_{k+1}$.

\begin{df}
A point $q \in C_0$ is called a singularity for $h$ if there is a continuous path $\tau :[0,1]\rightarrow C_0$ such that
\begin{enumerate}
\item $\tau(0)=p_0$ and $\tau(1)=q$.
\item $\forall \epsilon >0$, $h$ can be analytically continued along $\tau([0,1-\epsilon])$.
\item $h$ cannot be analytically continued along $\tau([0,1])$.
\end{enumerate}
\end{df}
The set of singularities could be, in principle, any subset of $C_0$. If it is the whole $C_0$, we say that $h$ has \emph{full singular set}.

There may also exist an open set $D\subset C_0$ containing $p_0$ such that for any path $\tau:[0,1]\rightarrow C_0$ with $\tau(0)=p_0$, $\tau(1)\in \partial D$ and $\tau([0,1[)\subset D$, $h$ can be analytically continued along $\tau([0,1-\epsilon])$ but not along $\tau([0,1])$. In the case where $\partial D$ is a topological disc, we say that $h$ has a \emph{natural boundary} for analytic continuation.

\begin{pro}\cite{CDFG}\label{proCDFG}
Let $\Sigma$ be a hyperbolic Riemann surface endowed with a branched projective structure. Let $\mathcal{D}$ be a developing map and $h$ be a germ of $\mathcal{D}^{-1}$.
\begin{enumerate}
\item If the projective structure is the one given by uniformization, then $h$  has a natural boundary for analytic continuation.
\item If the monodromy group is dense in $PSL(2,\mathbb{C})$, then $h$ has full singular set.
\end{enumerate}
\end{pro}

\begin{proof}
For a complete proof, see \cite{CDFG}. We give here some ideas of the proof because we think that the proof could be instructive for the comprehension of the proof theorem \ref{theoprincipal}.

\begin{enumerate}

\item In this case, the developing map is the inclusion $i:\mathbb{H}\hookrightarrow \mathbb{C}\mathbb{P}^1$. Then $\partial\mathbb{H}\subset \mathbb{C}\mathbb{P}^1$ is a natural boundary for analytic continuation of $h$.
 
\item Let $h$ be a germ of $\mathcal{D}^{-1}$ at $z_0\in\mathbb{C}\mathbb{P}^1$ and $p_0=h(z_0)$. The proof is based on the following lemma:
\begin{lem}\cite{CDFG}
For all $z \in \mathbb{C}\mathbb{P}^1$, there is a finite set $\mathcal{A}\subset\pi_1(\Sigma)$ and an infinite sequence $(\alpha_n)_{n\in \mathbb{N}^*}$ of elements of $\mathcal{A}$ which has the following properties: denoting $A_n=\alpha_1\alpha_2\cdots\alpha_n$ and $A_0=id$,
\begin{enumerate}
\item  the diameter of the ball 
\[B_n=\left\{w\in \mathbb{C}\mathbb{P}^1 \text{ such that } |(\rho(A_n))'(w)|\geq\frac{1}{2^n}\right\}\]
converges to $0$ exponentially fast when $n$ tends to infinity. 
\item For all $n\in \mathbb{N}$, $\rho(A_n)(\mathbb{C}\mathbb{P}^1-B_n)\subset D(z,\frac{cst}{2^n})$
\item For all $n\in \mathbb{N}$, neither $z_0$ nor $\rho(\alpha_{n})(z_0)$ belong to $B_{n-1}$
\end{enumerate}
\end{lem}
In this lemma (whose proof can be found in \cite{CDFG}), $\mathbb{C}\mathbb{P}^1$ is endowed with the standard spherical metric. In any of the two charts this metric is written: $|ds|=\frac{|dz|}{1+|z|^2}$. If $\gamma$ is a M\"{o}bius transformation, $\gamma'$ is the derivative of $\gamma$ and $|\gamma'(z)|$ is the spherical norm in $z$. If $z\in \mathbb{C}\mathbb{P}^1$ and $\alpha \in \mathbb{R}$, $D(z,\alpha)$ is the spherical disc of radius $\alpha$ centered at $z$. Let us prove that the previous lemma implies the proposition: with properties (a) and (c) of the previous lemma, one can construct for all $n\in \mathbb{N}$, a $C^{\infty}$ path $c_n:[0,1]\rightarrow\tilde{\Sigma}$ from $p_0$ to $\alpha_n(p_0)$, whose length is bounded by a constant independent of $n$ and such that for $n$ big enough $\mathcal{D}\circ c_n$ does not meet $B_{n-1}$. Then we define the path $c:[0,\infty[\rightarrow \tilde{\Sigma}$ as the infinite concatenation of paths $a_n:=A_{n-1}c_n$ (from $A_{n-1}(p_0)$ to $A_n(p_0)$). The $\rho$-equivariance gives:
\[\mathcal{D}\circ a_n=\rho(A_{n-1})\circ\mathcal{D}\circ c_n\]
As $\mathcal{D}\circ c_n$ does not meet $B_{n-1}$, we deduce, from property (a) of the previous lemma that the length of the path $\mathcal{D}\circ a_n$ converges exponentially fast to $0$ and so $\mathcal{D}\circ c(t)$ converges, when $t$ goes to infinity, toward a point in $\mathbb{C}\mathbb{P}^1$. Using property (b) of the previous lemma, this point is necessarily $z$ (because $\mathcal{D}\circ a_n\subset D(z,\frac{cst}{2^{n-1}})$). So $z$ is a singularity for analytic continuation of $h$. 
\end{enumerate}

\end{proof}

%Nous venons de montrer que pour une structure projective sur une surface de Riemann hyperbolique analytiquement finie d'application d?veloppante $\mathcal{D}$ et de groupe de monodromie $\Gamma$ dense dans $PSl(2,\mathbb{C})$; si $z_0\in \mathbb{C}\mathbb{P}^1$ et $h$ est un  germe d'holonomie d'une section locale de $\mathcal{D}$ en $z_0$, alors pour tout $z$ appartenant ? $\mathbb{C}\mathbb{P}^1$, il existe un chemin continu $b:[0,T]\rightarrow \mathbb{C}\mathbb{P}^1$ joignant $z_0$ ? $z$ ($b(0)=z_0$, $b(T)=z$) tel que $h$ se prolonge le long de $b([0,T-\epsilon])$ pour tout $\epsilon>0$ mais pas le long de $b([0,T])$. On aurait envie de g?n?raliser ce r?sultat ? presque toute trajectoire brownienne partant de $z_0$, c'est ? dire de montrer que pour presque tout $\omega \in \Omega_{z_0}$, il existe un temps d'arr?t $T(\omega)<+\infty$ tel que $h$ se prolonge le long de $\omega([0,T(\omega)-\epsilon])$ mais pas le long de $\omega([0,T(\omega)])$. Pour y arriver, il parait naturel, en restant dans l'esprit de la preuve de la proposition pr?c?dente, de montrer que pour presque toute trajectoire browniennne $\omega$ partant de $p_0$ dans dans $\tilde{\Sigma}$, il existe $z(\omega)\in \mathbb{C}\mathbb{P}^1$ tel que $\lim\limits_{t \to\infty} \mathcal{D}(\omega(t))=z(\omega)$. Or il n'en est rien comme le montre le th?or?me 

\section{Random walks.}\label{sectionrw}

In this section, after explaining some basic facts about random walks and stationary measures, we prove proposition \ref{propma} which is the key of the proof of theorems \ref{theoprincipal} and \ref{theo2}.

In this part, $\Gamma$ is a subgroup of $PSL(2,\mathbb{C})$ finitely generated and $\mu$ is a probability measure on $\Gamma$. $supp(\mu)$ is the support of $\mu$ and $<supp(\mu)>$ the group generated by $supp(\mu)$. Denote 
$\Omega=\Gamma^{\mathbb{N}^*}$, $\tau$ the $\sigma$-algebra generated by the cylinder sets in $\Omega$ and
$\mathbb{P}=\mu^{\mathbb{N}^*}$. The coordinate maps $h_i:\Omega\rightarrow\Gamma$ are $\mathbb{P}$-independent and identically distributed with law $\mu$. This part deals with the statistical behaviour of the action on $\mathbb{C}\mathbb{P}^1$ of the right random walk in $\Gamma$ with law $\mu$: $X_n(\omega)=h_1(\omega)\cdots h_n(\omega)$. 

The action of $\Gamma$ on $\mathbb{C}\mathbb{P}^1$ gives an action of $\Gamma$ on the set $\mathcal{P}(\mathbb{C}\mathbb{P}^1)$ of Borel probability measures on $\mathbb{C}\mathbb{P}^1$. If $\gamma\in\Gamma$, $\nu \in \mathcal{P}(\mathbb{C}\mathbb{P}^1)$ and $A$ is a Borel set in $\mathbb{C}\mathbb{P}^1$, this action is defined by: $\gamma\cdot\nu(A)=\nu(\gamma^{-1}(A))$.

We also define $\mu^{*n}:=\mu*\mu*\cdots\ast\mu$. The measure $\mu^{*n}$ on $\Gamma$ is the push forward of the product measure $\mu^{\otimes n}$ on $\Gamma^n$ by the map $\Gamma\times\cdots\times\Gamma\rightarrow\Gamma$,
 $(\gamma_1,\cdots,\gamma_n)\mapsto\gamma_1\cdots\gamma_n$. The law of $X_n$ is $\mu^{*n}$. If $\nu\in \mathcal{P}(\mathbb{C}\mathbb{P}^1)$, we also define the measure $\mu\ast\nu$ as the push forward on $\mathbb{C}\mathbb{P}^1$ of the product measure on $\Gamma\times\mathbb{C}\mathbb{P}^1$ by the map $\Gamma\times\mathbb{C}\mathbb{P}^1\rightarrow\mathbb{C}\mathbb{P}^1$, $(\gamma,x)\mapsto\gamma\cdot x$. So, if $A$ is a borel set in $\mathbb{C}\mathbb{P}^1$, we have:
\[\mu\ast\nu(A)=\sum_{\gamma\in\Gamma}\mu(\gamma)\nu(\gamma^{-1}(A))\]
and if f is a continuous function on $\mathbb{C}\mathbb{P}^1$:
\[\mu\ast\nu(f)=\sum_{\gamma\in\Gamma}\mu(\gamma)\int_{x\in\mathbb{C}\mathbb{P}^1}f(\gamma x)d\nu(x)\]

\begin{df}
The measure $\nu\in\mathcal{P}(\mathbb{C}\mathbb{P}^1)$ is said to be $\mu$-stationary if $\mu\ast\nu=\nu$, which means that for any Borel set $A$ in $\mathbb{C}\mathbb{P}^1$, we have: \[\sum_{\gamma\in\Gamma}\mu(\gamma)\nu(\gamma^{-1}(A))=\nu(A)\]
\end{df}

The following results are classical:
% One can refer to the fondateur article of Furstenberg.

\begin{teo}\label{theoma}[Furstenberg]
\begin{enumerate}
\item There always exists a $\mu$-stationary measure on $\mathbb{C}\mathbb{P}^1$ \cite{Fur}.
\item Let $\nu$ be a $\mu$-stationary measure on $\mathbb{C}\mathbb{P}^1$. Then, for almost every $\omega \in \Omega$, there is a measure $\lambda(\omega) \in \mathcal{P}(\mathbb{C}\mathbb{P}^1)$ such that the sequence of probability measures $X_n(\omega)\cdot\nu$ converges weakly toward $\lambda(\omega)$ \cite{Fur2}.
\item If $<supp(\mu)>$ is not an elementary group. Then, for almost every  $\omega\in\Omega$, there is $z(\omega)\in \mathbb{C}\mathbb{P}^1$ such that $\lambda(\omega)=\delta_{z(\omega)}$ (Dirac in $z(\omega)$) \cite{Fur2}.
\item If $<supp(\mu)>$ is not an elementary group, then a $\mu$-stationary measure on $\mathbb{C}\mathbb{P}^1$ is non-atomic \cite{Wo}.
\end{enumerate}
\end{teo}

\paragraph{The Lyapunov exponent. } The positivity of the Lyapunov exponent is a central result in the theory of random walks and is one of the key points of the proof of theorems \ref{theoprincipal} and \ref{theo2}.
\begin{teo}\label{Lyapunov}[Furstenberg]
If:
\begin{enumerate}
\item $\displaystyle{\int_\Gamma} \log||\gamma||\mathrm{d}\mu(\gamma) < +\infty$
\item $<supp(\mu)>$ is not an elementary group.
\end{enumerate}
Then, there exists $\lambda>0$ such that $\mathbb{P}$-almost surely, we have:
$$\frac{1}{n} \log || X_n || \longrightarrow \lambda.$$
\end{teo}

$\lambda$ is called the Lyapunov exponent of the random walk. The fact that $\frac{1}{n} \log || X_n ||$ converges almost surely to $\lambda\in[0,\infty[$ is a direct consequence of Kingman's subadditive ergodic theorem and requires the first hypothesis of the theorem ($\displaystyle\int_\Gamma \log||\gamma|| \mathrm{d}\mu(\gamma) < +\infty$). The fact that $\lambda>0$ requires the second hypothesis and was first proved by Furstenberg \cite[Theorem 8.6]{Fur} (see also \cite{BLa}).

\paragraph{A corollary of the positivity of the Lyapunov exponent. } 
In this paragraph, we prove a direct corollary of the positivity of the Lyapunov exponent which is well known from the specialists but not so easy to locate in the literature. First, we fix some notations: if $X=\begin{pmatrix}x_1\\x_2\end{pmatrix} \in \mathbb{C}^2-\begin{pmatrix}0\\0\end{pmatrix}$, then $[X]$ is the class of $X$ in $\mathbb{C}\mathbb{P}^1=\mathbb{C}^2-\begin{pmatrix}0\\0\end{pmatrix}/\mathbb{C}^*$. We have the following natural action:
\[\begin{array}{ccccc}
 & PSL(2,\mathbb{C}) \times \mathbb{C}\mathbb{P}^1 & \longrightarrow & \mathbb{C}\mathbb{P}^1 \\
 & \begin{pmatrix}a&b \\c&d\end{pmatrix},\begin{bmatrix}x_1\\x_2\end{bmatrix} & \longmapsto & \begin{bmatrix}ax_1+bx_2\\cx_1+dx_2\end{bmatrix} \\
\end{array}\]
We work with the following distance in $\mathbb{C}\mathbb{P}^1$ : if $X=\begin{pmatrix}x_1\\x_2\end{pmatrix}, Y=\begin{pmatrix}y_1\\y_2\end{pmatrix} \in \mathbb{C}^2-\begin{pmatrix}0\\0\end{pmatrix}$, $d([X],[Y])=\frac{|x_1y_2-y_1x_2|}{\sqrt{|x_1|^2+| x_2 |^2}\sqrt{| y_1 |^2+| y_2 |^2}}$. We denote $D(x,\alpha)$ the closed disc centered in $x$ with radius $\alpha$ and $(D(x,\alpha))^c$ its complementary set. If $g\in PSL(2,\mathbb{C})$, $||g||=\underset{|| X ||=1}{sup}|| gX ||$, where $|| X ||$ is the euclidean norm of the vector $X\in\mathbb{C}^2$. The goal of this part is to prove the following result: 

\begin{pro}\label{propma}
With the following hypothesis:
\begin{enumerate}
\item $\displaystyle{\int_\Gamma} \log||\gamma||\mathrm{d}\mu(\gamma) < +\infty$
\item $<supp(\mu)>$ is not an elementary group.
\end{enumerate}
Then, there are constants $0<\lambda'<\lambda''$ such that for $\mathbb{P}$-almost every $\omega\in\Omega$, there is $N(\omega)$ such that, for all $n>N(\omega)$, there are $y_n(\omega), z_n(\omega) \in \mathbb{C}\mathbb{P}^1$ such that:
\begin{enumerate}
\item $X_n(\omega)((D(y_n(\omega),e^{-\lambda'n}))^c)\subset D(z_n(\omega),e^{-\lambda'n})$
\item $X_n(\omega)(D(y_n(\omega),e^{-2\lambda''n}))\subset (D(z_n(\omega), \frac{1}{2}))^c$
\end{enumerate}
\end{pro}

\begin{rem}\label{remarque}
\begin{itemize}
\item Almost surely, the sequence $(z_n(\omega))$ defined in the previous proposition \ref{propma} converges to the point $z(\omega)$ defined in theorem \ref{theoma}.3. Indeed, let $\alpha$ be an accumulation point of the sequence $(z_n)$ different from $z$. Let $(n_i)_{i \in\mathbb{N}}$ such that $\lim\limits_{i \to\infty}z_{n_i}=\alpha$. Theorem \ref{theoma} gives $X_{n_i}\cdot\nu(D(\alpha,\frac{d(z,\alpha)}{2})\rightarrow\delta_{z}(D(\alpha,\frac{d(z,\alpha)}{2}))=0$. We deduce from proposition \ref{propma} that $\nu(D(y_{n_i},e^{-\lambda'n_i}))\rightarrow 1$. Extracting a new time, one can suppose that $y_{n_i}\rightarrow y \in \mathbb{C}\mathbb{P}^1$. Then $\nu(\{y\})=1$, which contradicts the fact that $\nu$ is a non-atomic measure.
\item The limit $z(\omega)$ of the sequence $(z_n(\omega))$ has also a dynamical interpretation: it is the projectivization of Oseledets' contracting direction of $X_n(\omega)^{-1}$. More precisely, when applying Oseledets's theorem to our situation (see \cite{Ar} or \cite[Theorem 1.5]{AB}), we get for almost every $\omega$, a one-dimensional vector space $F(\omega)$ in $\mathbb{C}^2$ (which depends measurably of $\omega$) such that:
$$ \lim\limits_{n\to\infty}\frac{1}{n}\log||X_n(\omega)^{-1}\cdot v||=\left\lbrace
\begin{array}{ll}
\lambda & \mbox{if $v\in\mathbb{C}^2-F(\omega)$}\\
-\lambda & \mbox{if $v\in F(\omega)-\{0,0\}$}
\end{array}
\right.$$
The point $z(\omega)\in\mathbb{C}\mathbb{P}^1$ is simply the projectivisation of the vectorial space $F(\omega)\subset\mathbb{C}^2$.
\end{itemize}
\end{rem}

\begin{proof}
From theorem \ref{Lyapunov}, there exists $\lambda>0$ such that $\mathbb{P}$-almost surely, we have:
$$\frac{1}{n} \log || X_n || \longrightarrow \lambda.$$

Fix $\lambda'$ and $\lambda''$ such that $0<\lambda'<\lambda<\lambda''$. 
$\mathbb{P}$-almost surely, we have for $n$ big enough:
\begin{eqnarray}
e^{\lambda'n}\leq || X_n|| \leq e^{\lambda''n} \label{eq3}
\end{eqnarray}
We know that if $g\in PSL(2,\mathbb{C})$, then there are $k, k'$ belonging to the projective special unitary group $PSU(2,\mathbb{C})$ and 
$a= \begin{pmatrix}
a_1&0 \\
0&a_2
\end{pmatrix}$ a diagonal matrix (with $| a_1|\geq| a_2|$) such that $g=kak'$. This is the so called Cartan decomposition.
Applying this decomposition to $X_n$, we obtain: 
$X_n=k_na_nk_n'$ with $k_n\in PSU(2,\mathbb{C})$, 
$a_n= \begin{pmatrix}
\alpha_n&0 \\
0&\alpha_n^{-1}
\end{pmatrix}$ 
and $| \alpha_n|\geq 1$.
As $|| a_n ||=|\alpha_n|$, and $k_n$ is norm-preserving and by the equation \eqref{eq3}, we have for $n$ big enough:
\begin{eqnarray}
e^{\lambda'n}\leq |\alpha_n| \leq e^{\lambda''n}
\end{eqnarray}

\begin{lem}
If $(e_1,e_2)$ is the canonical basis of $\mathbb{C}^2$ and $a= \begin{pmatrix}
\alpha&0 \\
0&\alpha^{-1}
\end{pmatrix}$ with $| \alpha |^2 > \sqrt{\frac{3}{2}}$
we have:
\begin{enumerate}
\item $d([X],[e_2])\geq| \alpha|^{-1}\Rightarrow d([aX],[e_1])\leq  | \alpha|^{-1}$
\item $d([X],[e_2])\leq | \alpha|^{-2}\Rightarrow d([aX],[e_1])\geq \frac{1}{2}$
\end{enumerate}
\end{lem}

\begin{proof}
Working in the chart $U_1=\{[X]=[x_1,x_2];x_1\neq0\}\rightarrow \mathbb{C},[x_1,x_2]\longmapsto\frac{x_2}{x_1}=z$, we have : \[d([X],[e_2])=\frac{1}{\sqrt{1+\mid z \mid^2}}\] So \[d([X],[e_2])\geq | \alpha|^{-1}\Leftrightarrow | z |^2\leq | \alpha|^{2}-1\] and \[d([X],[e_2])\leq | \alpha|^{-2}\Leftrightarrow | z |^2\geq | \alpha|^{4}-1\] A direct computation gives: 
\[d([aX],[e_1])=\frac{| \alpha|^{-2}| z|}{\sqrt{1+| \alpha|^{-4}| z|^2}}\]
 Let $\beta=| \alpha|^{-2}$ and $f(x)=\frac{\beta x}{\sqrt{1+\beta^{2} x^2}}$.
We have to prove the following:
\begin{enumerate}
\item $x\leq\sqrt{\frac{1}{\beta}-1}\Rightarrow f(x)\leq \sqrt{\beta}$
\item $x\geq \sqrt{\frac{1}{\beta^2}-1}\Rightarrow f(x) \geq \frac{1}{2}$
\end{enumerate}
$f'(x)=\frac{\beta}{(1+\beta^2 x^2)^{3/2}}>0$. So $f$ increases and so:
$x\leq\sqrt{\frac{1}{\beta}-1}\Rightarrow x\leq\sqrt{\frac{1}{\beta}}\Rightarrow f(x)\leq f(\sqrt{\frac{1}{\beta}})=\frac{\sqrt{\beta}}{\sqrt{1+\beta}}\leq\sqrt{\beta}$, which proves the first point. For the second one, we have: $x\geq\sqrt{\frac{1}{\beta^2}-1}\Rightarrow f(x)\geq f(\sqrt{\frac{1}{\beta^2}-1})=\frac{\sqrt{1-\beta^2}}{\sqrt{2-\beta^2}}\geq \frac{1}{2}$ for $\beta \leq \sqrt{\frac{2}{3}}$ (i.e. for $| \alpha |^2 \geq \sqrt{\frac{3}{2}}$).
\end{proof}

Using the previous lemma and the fact that an orthogonal transformation preserves the distance $d$, we conclude the proposition. Indeed, for $n$ big enough, we have:
\begin{align*}
a_n((D([e_2],e^{-\lambda'n}))^c)
&\subset a_n((D([e_2],|\alpha_n| ^{-1}))^c)\\
&\subset D([e_1],|\alpha_n| ^{-1})\\
&\subset D([e_1],e^{-\lambda'n})
\end{align*}
So:
\[a_nk_n'((D([k_n'^{-1}e_2],e^{-\lambda'n}))^c)\subset D([e_1],e^{-\lambda'n})\]
Consequently:
\[k_na_nk_n'((D([k_n'^{-1}e_2],e^{-\lambda'n}))^c)\subset D([k_ne_1],e^{-\lambda'n})\]
So, if we write: $y_n=k_n'^{-1}([e_2])$ and $z_n=k_n([e_1])$, we obtain for $n$ big enough:
\[X_n((D(y_n,e^{-\lambda'n}))^c)\subset D(z_n),e^{-\lambda'n})\]
The second assertion can be obtained by an analogous reasoning: for $n$ big enough,
\begin{align*}
a_n(D([e_2],e^{-2\lambda''n}))
&\subset a_n(D([e_2],|\alpha_n| ^{-2}))\\
&\subset (D([e_1],\frac{1}{2}))^c
\end{align*}
So, for $n$ big enough:
\[X_n(D(y_n,e^{-2\lambda''n}))\subset (D(z_n,\frac{1}{2}))^c\]

\end{proof}

\section{Brownian motion and discretization.}\label{sectiondiscret}

This part deals with the Brownian motion. Firstly, we recall the classical conformal invariance property of the Brownian motion. Secondly, we include a detailed treatment of the discretization procedure of Furstenberg-Lyons-Sullivan  which is close but not identical to that of Lyons and Sullivan (see \cite{LS} and \cite{BL}).

\subsection{Brownian motion and conformal invariance.}

Let $(M,g)$ be a connected Riemannian manifold with bounded geometry. The Brownian motion on $(M,g)$ is the diffusion process associated to the Laplace-Beltrami operator $\Delta$. It is defined on a probability space $(\Omega,\mathbb{P})$ and denoted by $(B_t)_{t\geq 0}$. We will make use of the following classical result of P.L\'{e}vy \cite{Le} which states that conformal maps are Brownian paths preserving up to a change of time-scale:

\begin{teo}\label{theoPL}[L\'{e}vy]
Let $(S_1,g_1)$ and $(S_2,g_2)$ be two connected, complete Riemannian surfaces and $f:(S_1,g_1)\rightarrow(S_2,g_2)$ be a conformal map.
Let $(B_t)_{t\geq 0}$ be a Brownian motion starting from a point $b_0 \in M_1$. Then, the process $(f(B_t))_{t\geq 0}$ is a changed time Brownian motion. In other words, there exists a family of strictly increasing functions $\sigma_{\omega}:[0,\infty[\longrightarrow  [0,+\infty[$ and a Brownian motion $(B'_s)_{s\geq 0}$ starting from $f(b_0)$ such that
\[f \circ B=B'\circ\sigma\]
\end{teo}

\begin{rems}
\begin{enumerate}
\item $f \circ B=B'\circ\sigma$ means: for all $\omega\in\Omega$ and for all $t\in[0,\infty[$, we have: $f \circ B_t(\omega)=B'_{\sigma_t(\omega)}(\omega)$.
\item The image Brownian motion $B'_s=f\circ B_{\sigma^{-1}(s)}$ is not necessarely defined for all positive times. It is defined for $s\in[0,T]$ where the stopping time $T:=\lim\limits_{t\to\infty}\sigma(t)$ is not necessarely equal to $\infty$.

\item If $| f'(z)|$ denotes the modulus of the derivative of $f$ in $z$ relatively to the metrics $g_1$ and $g_2$, the time-scale change is explicitly given by the following:
 \[\sigma_{\omega}(t)=\displaystyle{\int_0^t}| f'(B_u(\omega))|^2 du\]
\end{enumerate}
\end{rems}

\subsection{Discretization of the Brownian motion.}
In the most general context, this procedure associates to the Brownian motion in a Riemannian manifold $(M,g)$ a Markov chain in a discret $*$-recurrent set $X\subset M$ with time homogeneous transition probabilities. Here, we explain the discretization in the case where $M=\tilde{\Sigma}=\mathbb{D}$ is the universal covering space of a hyperbolic Riemann surface $\Sigma$ of finite type, and $X=\pi_1(\Sigma)\cdot 0$. We follow the presentation of \cite{KL}.\newline
Let $\Sigma$ be a hyperbolic Riemann surface of finite type. The fundamental group $\pi_1(\Sigma)$ of $\Sigma$ acts on $\tilde{\Sigma}$ ($=\mathbb{D}$), the universal covering space of $\Sigma$, by isometry for the Poincar\'{e} metric of the disc. For all $X \in\pi_1(\Sigma)$, we define: $F_X=X.\overline{D(0,\delta)}$ and $V_X=X.D(0,\delta')$ with $\delta<\delta'$. We also require that $\delta$ and $\delta'$ are small enough so that $F_X\cap V_{X'}=\varnothing$ for $X\neq X'$. Let $(\Omega_x,\mathbb{P}_x)$ be the set of Brownian paths starting from $x$ in $\mathbb{D}$ with the Wiener measure associated to the Poincar\'{e} metric in the hyperbolic disc. $\underset{X \in \pi_1(\Sigma)}{\bigcup}F_X$ is a recurrent set for the Brownian motion (because $\Sigma$ is of finite type). Let $X\in\pi_1(\Sigma)$. For $x\in F_X$, consider $\epsilon_x^{\partial V_X}$ the exit measure of $V_X$ for a Brownian motion starting from $x$. The Harnack constant $C_X$ of the couple $(F_X,V_X)$ is defined by:
\[C_X=\sup \left\{\frac{d\epsilon_x^{\partial V_X}}{d\epsilon_y^{\partial V_X}}(z); x,y \in F_X, z\in\partial V_X\right\}\]
where $\frac{d\epsilon_x^{\partial V_X}}{d\epsilon_y^{\partial V_X}}$ is the Radon-Nikodym derivative. Notice that, as elements of $\pi_1(\Sigma)$ act isometrically on $\mathbb{D}$, the Harnack constant of $(F_X,V_X)$ does not depend on $X \in \pi_1(\Sigma)$ (i.e. there is a constant $C$ such that for all $X\in\pi_1(\Sigma)$, $C_X=C$). Hence, the family of couples $(F_X,V_X)_{X\in\pi_1(\Sigma)}$ define system of L-S data in the sense of Ballmann-Ledrappier \cite[p 4]{BL}.\\

If $x\in V_{Id}$ and $\omega \in \Omega_x$, we define recursively:
\[S_0(\omega)=\inf\left\{t\geq0 ; \omega(t)\notin V_{Id}\right\}\]
and, for $n\geq1$:
\[R_n(\omega)=\inf\left\{t\geq S_{n-1}(\omega) ; \omega(t)\in \cup F_X \right\}\]
\[S_n(\omega)=\inf\left\{t\geq R_{n}(\omega) ; \omega(t)\notin \cup V_X \right\}\]

We also define $X_n(\omega)$ by:
 \[X_0(\omega)=Id \text{ and } w(R_n(\omega)) \in F_{X_n(\omega)} \text{ for }n\geq 1\]
\[\kappa_n(\omega)=\frac{1}{C}\left(\frac{d\epsilon_{X_n(\omega).0}^{\partial V_{X_n(\omega)}}}{d\epsilon_{\omega(R_n(\omega))}^{\partial V_{X_n(\omega)}}}(\omega(S_n(\omega)))\right)\] 
By definition of $C$ and $\kappa_n$, note that: $\frac{1}{C^2}\leq \kappa_n\leq 1$.\\

 Now, define $(\Omega_0\times [0,1]^{\mathbb{N}},\mathbb{P}_0\otimes leb^{\otimes \mathbb{N}})=(\tilde{\Omega},\tilde{\mathbb{P}})$. Let
\[\begin{array}{ccccc}
N_k: & \tilde{\Omega} & \longrightarrow & \mathbb{N}\\
 & (\omega,\alpha)=(\omega,(\alpha_n)_{n\in \mathbb{N}})=\tilde{\omega} & \longmapsto &  N_k(\tilde{\omega}) \\
\end{array}\]
 be the random variable defined recursively by:
\[N_0(\tilde{\omega})=0\]
\[N_k(\omega,\alpha)=\inf\left\{n>N_{k-1}(\omega,\alpha); \alpha_n<\kappa_n(\omega)\right\}\]
The following theorem is stated in \cite{LS} in the cocompact case but it is observed in \cite[Proposition 4]{K} that it is also valid in the general set-up:

\begin{teo}\cite[Theorem 6]{LS}\label{theodiscretisation}
The distribution law of $X_{N_1}$ defines a probability measure $\mu$ on $\pi_1(\Sigma)$ which satisfies for any Borel set $A$ in $\mathbb{D}$:
\[\tilde{\mathbb{P}}(X_{N_1}=x_1;...;X_{N_k}=x_k,\omega(S_{N_k})\in A)=\mu(x_1) \mu(x_1^{-1}x_2)...\mu(x_{k-1}^{-1}x_k)\epsilon_{x_k.0}^{\partial V_{x_k}}(A)\]
\end{teo}

\begin{cor}\cite{LS}
$(X_{N_k})_{k\in \mathbb{N}}$ is the realisation of a right random walk in $\pi_1(\Sigma)$ with law $\mu$, in other words  $(\gamma_{N_k}:=X_{N_{k-1}}^{-1}X_{N_k})_{k\in \mathbb{N}^*}$ is a sequence of independent, identically distributed random variables with law $\mu$.
\end{cor}

The two following propositions will be useful later:
\begin{pro}\cite[Corollaire 3.4]{KL}\label{propo}
 There is a constant $T>0$ such that almost surely $\frac{S_{N_k}}{k}$ converges to $T$ when $k$ goes to infinity.
\end{pro}

Note that there is a constant $D$ such that, $\forall X\in\pi_1(\Sigma)$ and $\forall z\in\partial F_X$, the Green function $G_{V_x}(X\cdot 0,z)=D$. This is because the Green function of a hyperbolic disc centered in $0$ is radial. Hence the L-S data $(F_X,V_X)_{X\in\pi_1(\Sigma)}$ are balanced (see definition in \cite[p9]{BL}). So we have:
\begin{pro}\cite[Theorem 3.2.b]{BL}\label{firstmoment}
The measure $\mu$ has full support and has a finite first moment with respect to the distance d associated to the Poincar\'{e} metric in $\mathbb{D}$, in other words $\displaystyle\int_{\gamma \in \pi_1(\Sigma)}d(\gamma\cdot 0,0)\,d\mu(\gamma)<+\infty$.
\end{pro}

\section{Proof of theorem \ref{theoprincipal}.}\label{sectionproof1}

Actually, we are going to prove the following theorem which is a reformulation of theorem \ref{theoprincipal} including the case where $\mathcal{D}$ is not onto. In this theorem, the Brownian motion in $\mathbb{D}$ (respectively $\mathbb{C}\mathbb{P}^1$) is the one associated to the hyperbolic metric (respectively any complete metric in its conformal class).

\begin{teo}\label{theoprincipal2}
Let $\Sigma$ be a Riemann surface of finite type endowed with a branched projective structure. Let $\mathcal{D}:\tilde{\Sigma}\rightarrow \mathbb{C}\mathbb{P}^1$ be a developing map and $\rho:\pi_1(\Sigma)\rightarrow PSL(2,\mathbb{C})$ be the monodromy representation associated to $\mathcal{D}$. Assume that $\rho$ is parabolic and non elementary. Let $(x_0,z_0)$ be a couple of points in $\tilde{\Sigma}\times\mathbb{C}\mathbb{P}^1$ such that $\mathcal{D}(x_0)=z_0$ and let $h$ be the germ of $\mathcal{D}^{-1}$ such that $h(z_0)=x_0$. 
\begin{description}
\item [First case: $\mathcal{D}$ is onto.]
Then the two following equivalent assertions are satisfied:
\begin{enumerate}
\item For almost every Brownian path $\omega$ starting from $x_0$, $\mathcal{D}(\omega(t))$ does not have limit when $t$ goes to $\infty$.
\item For almost every Brownian path $\omega$ starting from $z_0$, $h$ can be analytically continued along $\omega([0,\infty[)$.
\end{enumerate}
\item [Second case: $\mathcal{D}$ is not onto.]
Then the two following equivalent assertions are satisfied:
\begin{enumerate}
\item For almost every Brownian path $\omega$ starting from $x_0$, there is a point $z(\omega)$ such that $\lim\limits_{t \to\infty}\mathcal{D}(\omega(t))=z(\omega)$.
\item For almost every Brownian path $\omega$ starting from $z_0$, $h$ cannot be analytically continued along $\omega([0,\infty[)$.
\end{enumerate}
\end{description}

\end{teo}

\begin{proof}

Firstly, notice that according to remark \ref{rem1}, as the monodromy group $\Gamma:=\rho(\pi_1(\Sigma))$ is non elementary, $\Sigma$ is a hyperbolic Riemann surface.
Notice also that in any of the two cases ($\mathcal{D}$ onto and $\mathcal{D}$ not onto), the two conclusions are equivalent because of the conformal invariance of the Brownian motion. More precisely, in the first case, if $(B_t)_{t\in[0,\infty[}$ is a Brownian motion in $\tilde{\Sigma}$, then $(\mathcal{D}\circ B_{\sigma^{-1}(s)})_{0\leq s\leq T}$ is a Brownian motion in $\mathbb{C}\mathbb{P}^1$ stopped at time $T=\lim\limits_{t\to\infty}\sigma(t)$. If $\mathcal{D}\circ B_t$ does not have limit when $t$ goes to $\infty$, then almost surely $T=\infty$. Thus, almost surely, the germ $h$ of a local inverse of $\mathcal{D}$ can be analytically continued along the Brownian motion (defined for every positive time) $(\mathcal{D}\circ B_{\sigma^{-1}(s)})_{0\leq s\leq \infty}$. Conversely, if $h$ can be analytically continued along a generic Brownian path in $\mathbb{C}\mathbb{P}^1$, then almost surely $\mathcal{D}\circ B_t$ does not have limit when $t$ goes to $\infty$. Otherwise, we would have $T(\omega)<\infty$ for $\omega$ belonging to a set $A$ with strictly positive Wiener measure. Hence, for all $\omega\in A$, the germ $h$ could not be analytically continued along the Brownian path $(\mathcal{D}\circ B_{\sigma_{\omega}^{-1}(s)})_{0\leq s\leq T(\omega)}$. The proof of the equivalence of the two assertions in the second case (i.e. in the case where $\mathcal{D}$ is not onto) is similar.

\subsection{Proof in the case where $\mathcal{D}$ is onto.}

\paragraph{The discretization. }
In order to prove the theorem, we are going to use the discretization procedure explained in the previous part and the contraction property \ref{propma} proved in section \ref{sectionrw}. To simplify the notations, we take $x_0=0$ and $\omega \in \Omega_0$. If $\tilde{\omega}=(\omega,\alpha) \in \tilde{\Omega}$, then the path $\omega$ can be written as an infinite concatenation of paths:
\[\omega=\beta_0\ast\omega_0\ast\beta_1\ast\omega_1\ast\cdot\cdot\cdot\]
where $\beta_0=\omega_{|[0,S_{N_0}]}$, for $k\geq 0$, $\omega_k=\omega_{|[S_{N_k},R_{N_{k+1}}]}$ and for $k\geq 1$, $\beta_k=\omega_{|[R_{N_k},S_{N_k}]}$. Let $c_k(t)=X_{N_k}^{-1}\cdot\omega_k(t-S_{N_k})$. The $(c_k)_{k\in\mathbb{N}}$ form a family of portions of Brownian paths independent and identically distributed: the distibution law of their starting point is the exit measure of $V_{Id}=D(0,\delta')$ for a Brownian motion starting at $0$ and they are stopped at time $R_{N_{k+1}}-S_{N_k}$.
\[\omega=\beta_0\ast X_{N_0}c_0\ast\beta_1\ast X_{N_1}c_1\ast\cdot\cdot\cdot\]
Because of the $\rho$-equivariance, we have:
\[\mathcal{D}(\omega)=\mathcal{D}(\beta_0)\ast \rho(X_{N_0})\mathcal{D}(c_0)\ast\mathcal{D}(\beta_1)\ast \rho(X_{N_1})\mathcal{D}(c_1)\ast\cdot\cdot\cdot\]

Now, we are going to push forward the right random walk $X_{N_k}$ by $\rho$ in order to obtain a right random walk in the monodromy group $\Gamma$ and then apply proposition \ref{propma}. For this, we write $\tilde{\mu}=\rho_{\ast}\mu$ (where $\mu$ is the probability measure in $\pi_1(\Sigma)$ defined by the discretization procedure of the previous part) and $Y_{N_k}=\rho(X_{N_k})$. The process $(Y_{N_k})_{k\geq 0}$ is a realisation of a right random walk in $\Gamma$ with law $\tilde{\mu}$. The parabolicity of the monodromy representation implies the following (see \cite[Theorem 3.4.2]{A}) for a proof): there is a constant $a$ such that for all $\alpha \in \pi_1(\Sigma)$, we have $log(||\rho(\alpha)||)\leq a\cdot d(0,\alpha\cdot 0)$. We deduce, using proposition \ref{firstmoment}, that $\displaystyle\int_{\alpha \in \pi_1(\Sigma)}log(||\rho(\alpha)||)\,d\mu(\alpha)<+\infty$ and so $\displaystyle\int_{\gamma \in \Gamma}log(||\gamma||) d\tilde{\mu}(\gamma)<+\infty$. Then, the hypothesis of proposition \ref{propma} are satisfied. Consequently, there are $0<\lambda'<\lambda''$ such that for $\tilde{\mathbb{P}}$-almost every $\tilde{\omega}\in\tilde{\Omega}$, there is $N(\tilde{\omega})$ such that for all  $k>N(\tilde{\omega})$, there is $y_k(\tilde{\omega}), z_k(\tilde{\omega}) \in \mathbb{C}\mathbb{P}^1$ such that:
\begin{enumerate}
\item $Y_{N_k}\left((D(y_{k},e^{-\lambda'k}))^c\right)\subset D(z_{k},e^{-\lambda'k})$
\item $d\left(Y_{N_k}(D(y_{k},e^{-2\lambda''k})),z_{k}\right)\geq \frac{1}{2}$
\end{enumerate}

Then, the theorem follows from the following proposition:
\begin{pro}
\label{profond}
For almost every $\tilde{\omega}$, there is a sequence $(k_n)_{n \in \mathbb{N}}$ converging to infinity such that:
$$\mathcal{D}(c_{k_n})\cap D(y_{k_n},e^{-2\lambda''k_n})\neq \varnothing \text{ and } \mathcal{D}(c_{k_n})\cap (D(y_{k_n},e^{-\lambda'k_n}))^c \neq \varnothing.\\$$
\end{pro}

\paragraph{Proposition \ref{profond} implies Theorem \ref{theoprincipal2}. }
Indeed, by proposition \ref{propma}, the previous proposition implies that for an infinity of values of $k$, the portion $\rho(X_{N_k})\mathcal{D}(c_k)$ of the path $\mathcal{D}(\omega)$ visits $D(z_{k},e^{-\lambda'k})$ and $D(z_{k},\frac{1}{2})^c$, which proves that $\mathcal{D}(\omega(t))$ does not have limit when $t$ goes to infinity.\\

\paragraph{The technical lemma. }
We still have to prove proposition \ref{profond}. For that purpose, let us define:
\[E_k=\left\{\mathcal{D}(c_{k})\cap D(y_{k},e^{-2\lambda''k})\neq \varnothing\right\} \cap \left\{\mathcal{D}(c_{k})\cap (D(y_{k},e^{-\lambda'k}))^c \neq \varnothing\right\}\]

We have to prove that:

\begin{equation}\label{eq1}
 \tilde{\mathbb{P}}\left(\underset{n\geq0}{\cap}\underset{k\geq n}{\cup}E_k\right)=1 
\end{equation}

It turns out that there is a constant $c$ such that for all $k\in\mathbb{N}^*$, $\tilde{\mathbb{P}}(E_k)\geq\frac{c}{k}$ which implies $\sum\limits_{k\geq 1}\tilde{\mathbb{P}}(E_k)=\infty$. So, if the sequence $(E_k)_{k \in \mathbb{N}}$ were a sequence of independent events, one could conclude that \eqref{eq1} is true using Borel-Cantelli lemma. Unfortunately, one convinces easily that the $E_k$ are not independent: this is due to the fact that the $y_k$ are not mutually independent. This observation makes the proof of \eqref{eq1} more technical: instead of proving that $\tilde{\mathbb{P}}(E_k)\geq\frac{c}{k}$, we are going to prove the following lemma: 

\begin{lem}
\label{lemtechnique}
There exists constants $c>0$ and $N_0\in\mathbb{N}^*$ such that for all $N\geq N_0$ and $k>N$, we have:
$$\tilde{\mathbb{P}}\left(E_k|E_{k-1}^c,\cdots,E_N^c\right)\geq \frac{c}{k}.$$
\end{lem}
%\begin{equation}\label{eq2}
%\exists c>0 ,\, \exists N_0>0 \text{ such that } \forall N \geq N_0 ,\,\forall k>N , \,\tilde{\mathbb{P}}\left(E_k|E_{k-1}^c,...E_N^c\right)\geq \frac{c}{k} 
%\end{equation} 
\paragraph{Lemma \ref{lemtechnique} implies proposition \ref{profond}. }
Let us assume that Lemma \ref{lemtechnique} is proved. To prove proposition \ref{profond}, it is enough to prove \eqref{eq1}. So, it is enough to prove that $\forall N \in \mathbb{N}$,  $\tilde{\mathbb{P}}\left(\bigcap \limits _{n=N}^{\infty} E_n^c\right)=0$. Let $N \geq N_0$:
\[\tilde{\mathbb{P}}\left(\bigcap \limits _{n=N}^{\infty} E_n^c\right)=\underset{k\rightarrow\infty} {\lim}\tilde{\mathbb{P}}\left(\bigcap \limits _{n=N}^{k} E_n^c\right)\]

Let $k>N$, $u_k= \tilde{\mathbb{P}}\left(\bigcap \limits _{n=N}^{k} E_n^c\right)$ and $\alpha_k=\tilde{\mathbb{P}}\left(E_k^c|E_{k-1}^c,\cdots,E_N^c\right)$. We have:

\begin{eqnarray*}
  u_k & = & \alpha_k\cdot u_{k-1} \\
   & = & \alpha_k \alpha_{k-1} \cdots \alpha_{N+1}.u_N \\
   & \leq & \left(1-\frac{c}{k}\right)\left(1-\frac{c}{k-1}\right) \cdots \left(1-\frac{c}{N+1}\right).u_N \\
   & = & \prod \limits _{n=N+1}^{k}\left(1-\frac{c}{n}\right)\cdot u_N\\
   & \leq & \prod \limits _{n=N+1}^{k}e^{-\frac{c}{n}}\cdot u_N\\
   & = & \exp \left(-\sum \limits _{n=N+1}^{k}\frac{c}{n}\right)\cdot u_N 
     \underset{k \to \infty}{\longrightarrow}  0
\end{eqnarray*}

So, $\forall N>N_0$, $\tilde{\mathbb{P}}\left(\bigcap \limits _{n=N}^{\infty} E_n^c\right)=0$. And if $N<N_0$, we have $\bigcap \limits _{n=N}^{\infty} E_n^c \subset \bigcap \limits _{n=N_0}^{\infty} E_n^c$. So $\tilde{\mathbb{P}}\left(\bigcap \limits _{n=N}^{\infty} E_n^c\right)=0$, which finishes to prove \eqref{eq1}.\\

\paragraph{Proof of lemma \ref{lemtechnique}. }
We will need the following lemma:
% grossi?rement parlant expriment le fait suivant: si $\mathcal{D}$ est surjectif, alors pour un certain $r>0$, l'intersection du disque de centre $0$ et de rayon $r$  et de l'image r?ciproque par $\mathcal{D}$ d'un disque de rayon $e^{-2\lambda'' k}$ contient un disque dont rayon est de l'ordre de $e^{-2\lambda'' k}$. Plus pr?cis?ment:

\begin{lem}\label{lem1}
$\exists \beta>0,\, \exists r>0,\, \exists N_0 \in \mathbb{N}$ such that $\forall y \in \mathbb{C}\mathbb{P}^1,\, \exists x \in D(0,r)$ such that $\forall k\geq N_0$, we have:
\[D\left(x,\beta e^{-2\lambda''k}\right)\subset \mathcal{D}^{-1}\left(D(y,e^{-2\lambda''k})\right)\] 
\end{lem}

\begin{proof}
$\mathcal{D}$ is onto, so $\exists r>0$ such that $ \mathcal{D}(D(0,r))=\mathbb{C}\mathbb{P}^1 $. Let $\frac{1}{\beta}=\underset{D(0,2r)}{sup}\mid \mathcal{D}' \mid$. Let $N_0\in \mathbb{N}$ such that $\beta e^{-2\lambda''N_0}<r$. Let $y \in \mathbb{C}\mathbb{P}^1$ and let $x \in D(0,r)$ such that $\mathcal{D}(x)=y$. Let $k\geq N_0$ and $x_1 \in  D(x,\beta e^{-2\lambda''k})$. We have: 
$d(\mathcal{D}(x),\mathcal{D}(x_1))\leq \underset{D(x,\beta e^{-2\lambda''k})}{sup}|\mathcal{D}' | \cdot d(x,x_1)$. As $D(x,\beta e^{-2\lambda''k})\subset D(0,2r)$, we deduce that $d(\mathcal{D}(x),\mathcal{D}(x_1))\leq \frac{1}{\beta}.\beta. e^{-2\lambda''k}=e^{-2\lambda''k}$, which finishes the proof.
\end{proof}

Let us notice that, for $k$ big enough:
\[E_k=\left\{\mathcal{D}(c_{k})\cap D(y_{k},e^{-2\lambda''k})\neq \varnothing\right\}\]
Indeed, for $k$ big enough, the event $\mathcal{D}(c_{k})\cap (D(y_{k},e^{-\lambda'k}))^c \neq \varnothing$ is certain. To see this, note that: 
\[\left\{\mathcal{D}(c_{k})\cap (D(y_{k},e^{-\lambda'k}))^c \neq \varnothing\right\}=\left\{c_{k}\cap \mathcal{D}^{-1}(D(y_{k},e^{-\lambda'k}))^c \neq \varnothing\right\}=\varnothing\]
 Moreover, if $D$ is a compact disc in $\mathbb{D}$, then $\mathcal{D}^{-1}(D(y_{k},e^{-\lambda'k}))\cap D$ is a finite union of topological discs whose diameter converge to $0$ when $k$ goes to infinity, and the number of these discs is bounded by the degree of $\mathcal{D}_{|D}$. So, the sequence of continuous paths $c_k$ from $\partial V_{Id}$ to $\cup F_{\gamma}$ cannot, for $k$ big enough, be included in $\mathcal{D}^{-1}(D(y_{k},e^{-\lambda'k}))$.

Let $N \in \mathbb{N}$ big enough and $k>N$. Write $D_k(\tilde{\omega})=\mathcal{D}^{-1}(D(y_{k}(\tilde{\omega}),e^{-2\lambda''k}))$. We are going to prove the following lemma:
\begin{lem}
\[\tilde{\mathbb{P}}\left(E_k\,|\,E_{k-1}^c,\cdots,E_N^c\right)\geq \underset{x\in D(0,r)}{inf}\tilde{\mathbb{P}}\left(\{c_{0}\cap D(x,\beta e^{-2\lambda''k})\neq \varnothing\}\right)\]
where $r$ and $\beta$ are given in lemma \ref{lem1}.
\end{lem}

\begin{proof}
From the proof of proposition \ref{propma}, we see that, by construction, $y_k$ depends only on the set $X_{N_1}$,...,$X_{N_k}$ (i.e. it depends on the set $\gamma_{N_1}$,...,$\gamma_{N_{k}}$) and $c_k$ depends only on $X_{N_k}^{-1}X_{N_{k+1}}=\gamma_{N_{k+1}}$. As the $\gamma_{N_i}$ are mutually independent, we deduce that $y_k$ and $c_k$ are independent. Thus, we have:
\[\tilde{\mathbb{P}}\left(E_k\,|\,E_{k-1}^c,\cdots,E_N^c\right)\] 
\begin{align*}
 \geq \underset{y \in \mathbb{C}\mathbb{P}^1}{inf} \,\tilde{\mathbb{P}}\Big(&\{c_k \cap \mathcal{D}^{-1}(D(y,e^{-2\lambda''k})) \neq \varnothing\} \,|\,\\
 & c_{k-1} \cap D_{k-1} =\varnothing,\cdots,c_{N} \cap D_{N} =\varnothing\}\Big).
 \end{align*}

\[ = \underset{y \in \mathbb{C}\mathbb{P}^1}{inf}\, \tilde{\mathbb{P}}\left(\{c_k \cap \mathcal{D}^{-1}(D(y,e^{-2\lambda''k})) \neq \varnothing\}\right)\]

(this is because the event $\left\{c_{k-1} \cap D_{k-1} =\varnothing,\cdots,c_{N} \cap D_{N} =\varnothing \right\}$ and the event $\left\{c_k \cap \mathcal{D}^{-1}(D(y,e^{-2\lambda''k})) \neq \varnothing\right\}$ are independent)

\[  \geq  \underset{x\in D(0,r)}{inf}\tilde{\mathbb{P}}\left(\{c_{k}\cap D(x,\beta e^{-2\lambda''k})\neq \varnothing\}\right) \]

(this last inequality comes from lemma \ref{lem1})

\[=\underset{x\in D(0,r)}{inf}\tilde{\mathbb{P}}\left(\{c_{0}\cap D(x,\beta e^{-2\lambda''k})\neq \varnothing\}\right) \]

(because the paths $c_k$ are i.i.d.)
\end{proof}

So, we still have to prove the following:
\begin{lem}
\label{lemmme}
There is a constant $c$ such that for $k$ big enough, we have:
\[\underset{x\in D(0,r)}{inf}\tilde{\mathbb{P}}\left(\{c_{0}\cap D(x,\beta e^{-2\lambda''k})\neq \varnothing\}\right)\geq \frac{c}{k}\]
\end{lem}

The proof of this fact is a little bit technical. So, we start with the general idea: we will prove that the value of $\underset{x\in D(0,r)}{inf}\tilde{\mathbb{P}}\left(\{c_{0}\cap D(x,\beta e^{-2\lambda''k})\neq \varnothing\}\right)$ is almost the same as the probability that a Brownian path in $\mathbb{C}$ (with Euclidean metric) starting from $z=\frac{1}{2}$ would reach $D(0,e^{-k})$ before reaching $\partial D(0,1)$. Using Brownian invariance by the exponential map, this probability is equal to the probability that a plane Brownian motion starting from $z=-\log 2$ would reach the line $x=-k$ before reaching the line $x=0$. As the two canonical coordinates of a plane Brownian motion are one-dimensional Brownian motions, the previous probability is equal to $\mathbb{P}_{-\log 2}(T_{-k}\leq T_0)$ (the probability that a Brownian motion in $\mathbb{R}$ starting from $-\log 2$ would reach the point $-k$ before reaching the point $0$). For all $x\in[-k;0]$, the map $f(x)=\mathbb{P}_x(T_{-k}\leq T_0)$ is harmonic and satisfies $f(-k)=1$, $f(0)=0$. We deduce that $f(x)=-\frac{x}{k}$. Hence the desired probability is $f(-\log 2)=\frac{\log 2}{k}$.\newline

Let us give a precise proof. Recall that $\mathbb{P}_y$ is the Wiener measure of the Brownian motion starting from $y$ (Brownian motion associated to the Poincar\'{e} metric of the disc if $y$ belongs to the Poincar\'{e} disc and associated to the Euclidean metric if $y$ belongs to $\mathbb{C}$). Denote $\mathbb{P}_{m}:=\int \mathbb{P}_y dm(y)$ where $m$ is the exit measure of $V_{Id}=D(0,\delta')$ for a Brownian path starting from $0$. For a closed set $A$, and a Brownian path $\omega$, denote $T_A(\omega)$ the reaching time of the set $A$. Let $\epsilon>0$ and $x \in D(0,r)$. Choose $\gamma\in\pi_1(\Sigma)$ such that $F_{\gamma}\cap D(x,\epsilon)=\varnothing$. Then, we have:

\[\underset{x\in D(0,r)}{inf}\tilde{\mathbb{P}}\left(\{c_{0}\cap D(x,\beta e^{-2\lambda''k})\neq \varnothing\}\right)\]
  \[\geq\tilde{\mathbb{P}}\left(\{c_{0}\cap D(x,\beta e^{-2\lambda''k}) \neq \varnothing \}\cap \{\gamma_{N_1}=\gamma\}\right)\]
	As the event $\{\gamma_{N_1}=\gamma\}=\{X_{N_1}=\gamma\}$ contains the event $\{N_1=1\}\cap\{X_1=\gamma\}$, we deduce that the previous probability is greater than:
   \[\tilde{\mathbb{P}}\left(\{c_{0}\cap D(x,\beta e^{-2\lambda''k}) \neq \varnothing \}\cap \{X_{1}=\gamma\}\cap \{N_1=1\}\right)\]
  \[\geq\tilde{\mathbb{P}}\left(\{c_{0}\cap D(x,\beta e^{-2\lambda''k}) \neq \varnothing \}\cap \{X_{1}=\gamma\}\cap \{\alpha_1\leq\frac{1}{C^2}\}\right)\]
  %\[\geq \mathbb{P}_{\mu}\otimes\lambda(T_{D(x,\beta e^{-2\lambda''k})}\leq T_{\cup F_{\alpha}}\cap T_{F_{\gamma}}\leq T_{\cup F_{\alpha}}\cap \alpha_1<\frac{1}{C^2})\]
  \[= \frac{1}{C^2}\cdot\mathbb{P}_m\left(\{T_{D(x,\beta e^{-2\lambda''k})}\leq T_{\cup F_{\alpha}}\}\cap \{T_{F_{\gamma}}\leq T_{\cup F_{\alpha}}\}\right)\]
  If $k$ is big enough so that $\beta e^{-2\lambda''k}<\frac{\epsilon}{2}$, then by the strong Markov property, the last quantity is
  \begin{align*}
  \geq \frac{1}{C^2}\cdot&\mathbb{P}_m\left(T_{D(x,\frac{\epsilon}{2})}\leq T_{\cup F_{\alpha}}\right)\cdot\underset{y \in \partial D(x,\frac{\epsilon}{2})}{inf}\mathbb{P}_y\left(T_{D(x,\beta e^{-2\lambda''k})}\leq T_{\partial D(x,\epsilon)}\right)\\
  &\cdot\underset{z \in \partial D(x,\epsilon)}{inf}\mathbb{P}_z\left(T_{F_{\gamma}}\leq T_{\cup F_{\alpha}}\right)
  \end{align*}
 As $x \in D(0,r)$, $\exists a>0$ (which does not depend on  $x$) such that:
 \[\mathbb{P}_{m}\left(T_{D(x,\frac{\epsilon}{2})}\leq T_{\cup F_{\alpha}}\right)\cdot\underset{z \in \partial D(x,\epsilon)}{inf}\mathbb{P}_z\left(T_{F_{\gamma}}\leq T_{\cup F_{\alpha}}\right)\geq a\]
 
 \begin{lem}
 $\exists b>0$ (which does not depend on $x$) such that:
 \[\forall y \in \partial D(x,\frac{\epsilon}{2}),\, \mathbb{P}_y\left(T_{D(x,\beta e^{-2\lambda''k})}\leq T_{\partial D(x,\epsilon)}\right)\geq \frac{b}{k}\]
 \end{lem}
 
 \begin{proof}
 For $p\in\mathbb{C}$, denote $D_{eucl}(p,\alpha)$ the disc with centre $p$ and radius $\alpha$ in $\mathbb{C}$ for the Euclidean metric. Let $y \in \partial D(x,\frac{\epsilon}{2})$. There are constants $c_1>0$, $0<c_2<1$ such that, for $k$ big enough, there is a biholomorphism $\Psi_k$ which identifies:
 \begin{itemize}
 \item $D(x,\beta e^{-2\lambda''k})$ and $D_{eucl}(0,c_1 e^{-2\lambda''k}):=D_1(k)$
 \item $D(x,\frac{\epsilon}{2})$ and $D_{eucl}(0,c_2):=D_2$
 \item $D(x,\epsilon)$ and $D_{eucl}(0,1):=D_3$
 \item $y$ and $c_2$ 
 \end{itemize}
 By the conformal invariance of the Brownian motion:
 \[\mathbb{P}_y\left(T_{D(x,\beta e^{-2\lambda''k})}\leq T_{\partial D(x,\epsilon)}\right)=\mathbb{P}_{c_2}\left(T_{D_1(k)}\leq T_{\partial D_3}\right)\]
 The exponential map sends:
 \begin{itemize}
 \item the line $\Delta_1(k):=\{x=\log(c_1 e^{-2\lambda''k})\}$ onto $\partial D_1(k)$,
 \item the line $\Delta_2:=\{x=\log(c_2)\}$ onto $\partial D_2$,
 \item the line $\Delta_3:=\{x=0\}$ onto $\partial D_3$.
 \end{itemize}
 So, by the conformal invariance of the Brownian motion, there is a constant $b$ such that for $k$ big enough:
 \[\mathbb{P}_{c_2}\left(T_{D_1(k)}\leq T_{\partial D_3}\right)=\mathbb{P}_{\log(c_2)}\left(T_{\Delta_1(k)}\leq T_{\Delta_3}\right)=\frac{-\log(c_2)}{2\lambda''k-\log(c_1)}\geq\frac{b}{k}\]

 \end{proof}
 
 So we found a constant $c=\frac{ab}{C^2}$ such that for $k$ big enough, we have $\underset{x\in D(0,r)}{inf}\tilde{\mathbb{P}}\left(\{c_{0}\cap D(x,\beta e^{-2\lambda''k})\neq \varnothing\}\right)\geq \frac{c}{k}$. This ends the proof of lemma \ref{lemmme} and that of the theorem.

\begin{rem} \label{remPSU2}
In theorem \ref{theoprincipal}, we made the assumption that $\Gamma$ is non elementary (this assumption was necessary to get the positivity of the Lyapounov exponent).
Remark that the conclusion holds if $\Gamma$ is conjugate to a subgroup of $PSU(2,\mathbb{C})$. Indeed, in this case, there exists $k_1>0$ such that for amost every $\tilde{\omega} \in \Omega$, for all $n\in\mathbb{N}$ the path $\mathcal{D}(c_n(\tilde{\omega}))$ contains two points at spherical distance greater than $k_1$. As a group conjugated to a subgroup of $PSU(2,\mathbb{C})$ quasi-preserves the spherical metric, we have:
\[\exists k_2>0 \text{ such that } \forall \gamma \in \Gamma, \forall z,z'\in\mathbb{C}\mathbb{P}^1, d(\gamma\cdot z,\gamma\cdot z')>k_2.d(z'z')\]
So almost surely, for all $n\in\mathbb{N}$, the path $\mathcal{D}(\omega_n)=Y_n\cdot\mathcal{D}(c_n)$ contains two points at distance greater than $k_1.k_2$. So almost surely, $\mathcal{D}(\omega(t))$ does not have limit when $t$ goes to infinity.
\end{rem}
	
\subsection{Proof in the case where $\mathcal{D}$ is not onto.}

Let $(x_0,z_0)$ be a couple of points in $\tilde{\Sigma}\times\mathbb{C}\mathbb{P}^1$ such that $\mathcal{D}(x_0)=z_0$ and let $h$ be the germ of $\mathcal{D}^{-1}$ satisfying $h(z_0)=x_0$. We are going to prove that for almost every Brownian path $\omega$ starting from $z_0$, the germ $h$ cannot be analytically continued along $\omega([0,\infty[)$.

Let $U$ be the open set in $\mathbb{C}\mathbb{P}^1$ defined by $U:=\mathcal{D}(\tilde{\Sigma})$. Its complementary $U^c$ is a closed $\Gamma$-invariant set (infinite because $\Gamma$ is not elementary). As $\Gamma$ is non elementary, we are in one of the following situations (see \cite[Paragraph 1]{S} for a proof):
\begin{enumerate}
\item either $\Gamma$ is dense in $PSL(2,\mathbb{C})$,
\item or $\Gamma$ is discrete,
\item or, replacing $\Gamma$ by a subgroup of index $2$ if necessary, $\Gamma$ is conjugate to a dense subgroup of $PSL(2,\mathbb{R})$.
\end{enumerate}
Case $1.$ is impossible because $\Gamma$ leaves invariant the closed set $U^c\neq \mathbb{C}\mathbb{P}^1$. In case $2.$, $\Gamma$ is Kleinian. The limit set $\Lambda(\Gamma)$ being the smallest closed  $\Gamma$-invariant set of $\mathbb{C}\mathbb{P}^1$, we have $\Lambda(\Gamma)\subset U^c$. As $\Gamma$ is non elementary, a theorem of Myrberg \cite{My} (see also \cite{Do}) asserts that the logarithmic capicity of $\Lambda(\Gamma)$ is strictly positive. Hence  $\Lambda(\Gamma)$ (and so $U^c$) is visited by the Brownian motion in finite time, which implies that $h$ cannot be analytically continued along a generic Brownian path. In case $3.$, $U^c$ contains a Jordan curve. So $U^c$ is also visited by the Brownian motion.

\end{proof}

\section{Analytic continuation of holonomy germs of algebraic foliations.}\label{sectionRiccati}

\subsection{Riccati foliations and branched projective structures.}

Let $(\Pi,M,X,\mathcal{F})$ be a Riccati foliation (see the definition in the introduction). Using the transversality of a generic fibre with $\mathcal{F}$, we can define a monodromy representation associated to such foliations: denote $\{x_1,\cdots,x_n\}$ the points in $X$ such that the fiber over $x_i$ is an invariant line. Denote $\Sigma=X-\{x_1,\cdots,x_n\}$. Fix $x_0\in X$. Let $\alpha:[0,1]\rightarrow \Sigma$ be a closed curve in $\Sigma$ based in $x_0$. Let $z \in \Pi^{-1}(x_0):=F_{x_0}$. There is an unique path $\tilde{\alpha}:[0,1]\rightarrow M$ lifting $\alpha$, belonging to the leaf through $z$ and satisfying $\tilde{\alpha}(0)=z$. The map $z\mapsto\phi_{\alpha}(z)=\tilde{\alpha}(1)$ is a biholomorphism of $F_{x_0}$ which only depends on the homotopy class of $\alpha$. Then, a local trivialisation of the fibre-bundle around $x_0$ gives an identification $F_{x_0}\cong \mathbb{C}\mathbb{P}^1$ and we obtain a representation:
$$\rho:\pi_1(\Sigma,x_0)\longrightarrow PSL(2,\mathbb{C}).$$
called \emph{monodromy representation of the foliation}.
Take any holomorphic section $s:X\to M$ not invariant by the foliation (recall that such a section always exists, see remark \ref{rem-section}). We can transport by the foliation the unique complex projective structure on $F_{x_0}$ (or on any other non invariant fiber). We obtain a branched complex projective structure on $S:=s(\Sigma)\cong \Sigma$ whose monodromy representation is the monodromy representation of the foliation (the branched points are the points of $S$ where the foliation is tangent to $S$). By definition, if $p\in S$ is not a branched point and if $h:(F_{x_0},p_0)\to (S,p)$ is a holonomy germ of the foliation, the analytic continuation of $h^{-1}$ defines a developing map of the complex projective structure on $S$. 

We have just explained how to pass from a Riccati foliation to a complex projective structure. Reciprocally, starting from a parabolic branched complex projective structure on a Riemann surface $\Sigma$ of finite type, we can obtain a Riccati foliation after suspending the representation and compactifying with local models as explained briefly in the introduction (see also \cite{DD} or \cite{CDFG}).

\subsection{Proof of theorem \ref{theoRiccati}.}

Item $1.$ is a direct application of theorem \ref{theoprincipal}.
Proof of item $2$ goes as follows: let $(s_i)_{i=0,1}$ be two sections of $\Pi$, $S_i=s_i(\Sigma)$ and $\overline{S_i}=s_i(X)$ . Let $g_1$ be a complete metric on $\overline{S_1}$ in its conformal class. Let $h:(\overline{S_1},p_1)\to(\overline{S_0},p_0)$ be a holonomy germ. We want to prove that $h$ can be analytically along a Brownian motion $(B_t)_{t\geq 0}$ (with respect to the metric $g_1$) starting at $p_1$. First, using the strong Markov property, one can assume that $p_1\in S_1$. Moreover, $(B_t)$ does not visit the points $\{x_1,\cdots,x_n\}$. % Secondly, denoting by $hyp_1$ the hyperbolic metric on $S_1$. The inclusion $i:(S_1,hyp_1)\hookrightarrow (\overline{S_1},g_1)$ is a conformal map. Hence, it is Brownian path preserving with a time reparametrisation $\sigma$ satisfying almost surely: $\lim\limits_{t\to\infty}\sigma(t)=\infty$. So, instead of considering the Brownian motion with respect to the metric $g_1$, one can consider it with respect to the metric $hyp_1$, which will allow to apply theorem \ref{theoprincipal}. 
Secondly, $h$ can be written $h=\mathcal{D}_0^{-1}\circ\mathcal{D}_1$ where, for $i\in\{0,1\}$, $\mathcal{D}_i$ is a developing map associated to the branched projective structure on $S_i$. By the conformal invariance of the Brownian, after time reparametrization, $\mathcal{D}_1\circ B_t$ is a Brownian motion in $\mathbb{C}\mathbb{P}^1$ along which $\mathcal{D}_2^{-1}$ can be analytically continued by Item $1$. This concludes the proof.

\section{Proof of theorem \ref{theo2}.}\label{sectionproof2}

Let $\Sigma$ be a hyperbolic Riemann surface of finite type. Let $\mathcal{D}:\tilde{\Sigma}=\mathbb{D}\rightarrow \mathbb{C}\mathbb{P}^1$ and $\rho:\pi_1(\Sigma)\rightarrow PSL(2,\mathbb{C})$ be a couple developing map-monodromy representation associated to a branched complex projective structure on $\Sigma$. Assume that this structure is parabolic type and $\rho$ is non-elementary. We have to prove that for almost every Brownian path $\omega$ starting at $0\in\mathbb{D}$, there exists $z(\omega) \in \mathbb{C}\mathbb{P}^1$ such that:
$$\frac{1}{t}\cdot \displaystyle\int_0^t\delta_{\mathcal{D}(\omega(s))}\cdot ds\underset{t\to\infty}\longrightarrow\delta_{z(\omega)}.$$

As in the proof of the previous theorem, we are going to use the discretization procedure of Furstenberg, Lyons, Sullivan. Nevertheless, the notations are slightly modified. If $\tilde{\omega}=(\omega,\alpha) \in \tilde{\Omega}$, the infinite path $\omega$ can be written as an infinite concatenation of paths:
\[\omega=\beta_0\ast\omega_0\ast\omega_1\ast\cdot\cdot\cdot\]
where $\beta_0=\omega_{|[0,S_{N_0}]}$ and for $k\geq 0$, $\omega_k=\omega_{|[S_{N_k},S_{N_{k+1}}]}$. For $k\geq 0$, we define $c_k:=X_{N_k}^{-1}\cdot\omega_k$. Then we have:
\[\omega=\beta_0\ast X_{N_0}c_0\ast X_{N_1}c_1\ast\cdot\cdot\cdot\]
Using $\rho$-equivariance of $\mathcal{D}$, we have:
\[\mathcal{D}(\omega)=\mathcal{D}(\beta_0)\ast \rho(X_{N_0})\mathcal{D}(c_0)\ast \rho(X_{N_1})\mathcal{D}(c_1)\ast\cdot\cdot\cdot\]

The sequence of random variables  $X_{N_k}$ is a realisation of a right random walk in $\pi_1(\Sigma)$ whith law $\mu$ and the sequence $Y_{N_k}=\rho(X_{N_k})$ is a realisation of a right random walk in $\rho(\pi_1(\Sigma))$ whith law $\tilde{\mu}=\rho_*\mu$. Let $y_k(\tilde{\omega})$ and $z_k(\tilde{\omega})$ be the two sequences of random points in $\mathbb{C}\mathbb{P}^1$ defined in proposition \ref{propma}. According to remark \ref{remarque}, almost surely $z_k(\tilde{\omega})\rightarrow z(\tilde{\omega})$.

In order to prove the theorem, it is enough to prove that for $\tilde{\mathbb{P}}$-almost every $\tilde{\omega}=(\omega,\alpha) \in \tilde{\Omega}$, for all $\epsilon>0$, we have:
\begin{equation}
\lim\limits_{t \to \infty}\frac{1}{t}\cdot leb\left\{u \in[0,t]\text{ such that } \mathcal{D}(\omega(u)) \in D(z(\tilde{\omega}),\epsilon)\right\}=1
\label{equ1}
\end{equation}

For $k\geq 0$, denote:
\[T_k(\tilde{\omega})=leb\left\{t \in [S_{N_k},S_{N_{k+1}}] \text{ such that } \mathcal{D}(c_k(\tilde{\omega})(t)) \in D(y_k(\tilde{\omega}),e^{-\lambda' k})\right\}\]

We have the following:

\begin{pro}\label{pro1}
Almost surely $\lim\limits_{k \to \infty} T_k=0$.
\end{pro}

Before proving this proposition, let us show why this implies the theorem.
First, if we assume that almost surely $\lim\limits_{k \to \infty} T_k=0$, then we have almost surely: 
\begin{equation}
\lim\limits_{n \to \infty} \frac{1}{n}\sum_{k=0}^{n-1}T_k=0
\label{equ2}
\end{equation}

Now, according to proposition \ref{propo}, there is a constant $T$ such that, almost surely: 
\begin{equation}
\lim\limits_{n\rightarrow\infty}\frac{S_{N_n}}{n}=T
\label{equ3}
\end{equation}
Let $\tilde{\omega}$ belonging to the full measure set where equations \eqref{equ2} and \eqref{equ3} are satisfied. Let $\epsilon$ be a strictly positive real. According to remark \ref{remarque}, $z_k(\tilde{\omega})\rightarrow z(\tilde{\omega})$. So, there is $I_0(\tilde{\omega})$ such that $\forall k \geq I_0$, we have $D(z_k(\tilde{\omega}),e^{-\lambda'k})\subset D(z(\tilde{\omega}),\epsilon)$. So:
\[\frac{1}{n}\sum_{k=0}^{n-1}T_k(\tilde{\omega})\longrightarrow 0\]
\[\Longrightarrow \frac{1}{S_{N_{n}}}\sum_{k=0}^{n-1}T_k(\tilde{\omega})\longrightarrow 0\]
\[\Longrightarrow \frac{1}{S_{N_{n}}}\sum_{k=I_0}^{n-1}leb\left\{t \in [S_{N_k},S_{N_{k+1}}] \text{ such that } \mathcal{D}(c_k(\tilde{\omega})(t)) \in D(y_k(\tilde{\omega}),e^{-\lambda' k})\right\}\longrightarrow 0\]
which implies by proposition \ref{propma} 
\[\Longrightarrow \frac{1}{S_{N_{n}}}\sum_{k=I_0}^{n-1}leb\left\{t \in [S_{N_k},S_{N_{k+1}}] \text{ such that } \mathcal{D}(\omega(t)) \notin D(z_k(\tilde{\omega}),e^{-\lambda' k})\right\}\longrightarrow 0\]
\[\Longrightarrow \frac{1}{S_{N_{n}}}\sum_{k=I_0}^{n-1}leb\left\{t \in [S_{N_k},S_{N_{k+1}}] \text{ such that } \mathcal{D}(\omega(t)) \notin D(z(\tilde{\omega}),\epsilon)\right\}\longrightarrow 0\]
\[\Longrightarrow \frac{1}{S_{N_{n}}}\sum_{k=1}^{n-1}leb\left\{t \in [S_{N_k},S_{N_{k+1}}] \text{ such that } \mathcal{D}(\omega(t)) \notin D(z(\tilde{\omega}),\epsilon)\right\}\longrightarrow 0\]
\[\Longrightarrow \frac{1}{S_{N_{n}}}.leb\left\{t \in [0,S_{N_{n}}] \text{ such that } \mathcal{D}(\omega(t)) \notin D(z(\tilde{\omega}),\epsilon)\right\})\longrightarrow 0\]
\[\Longrightarrow \frac{1}{S_{N_{n}}}.leb\left\{t \in [0,S_{N_{n}}] \text{ such that } \mathcal{D}(\omega(t)) \in D(z(\tilde{\omega}),\epsilon)\right\})\longrightarrow 1\]
\[\Longrightarrow \frac{1}{n}.leb\left\{t \in [0,n] \text{ such that } \mathcal{D}(\omega(t)) \in D(z(\tilde{\omega}),\epsilon)\right\})\longrightarrow 1\Longrightarrow \eqref{equ1}\]

\paragraph{Beginning of the proof of proposition \ref{pro1}:}
Using Borel-Cantelli, it is enough to prove for all $\epsilon>0$ that: $\tilde{\mathbb{P}}(T_k\geq\epsilon)\leq \frac{2}{k^2}$.
Let $K$ be a positive constant (to be determined later) and denote:
$$A_k=\left\{\tilde{\omega}\text{ s.t. } c_0(\tilde{\omega})\cap D(0,K \log(k))^c\neq \varnothing\right\}$$
%$$B_k=\left\{\gamma\in\Pi_1{\Sigma}\text{ s.t. } d_{hyp}(\gamma.0;0)\leq K\log k\right\}$$
%As in the proof of theorem 1, $m$ is the uniform measure on $\partial D_{hyp}(0,\delta')$ and $\mathbb{P}_m$ is the Wiener measure with initial distribution $m$.
We are going to prove the following:

\begin{equation}
\label{equ4}
\tilde{\mathbb{P}}(T_k\geq\epsilon)\leq \tilde{\mathbb{P}}(A_k) + \sup_{y \in \mathbb{C}\mathbb{P}^1} \mathbb{P}_0(\tau_{y,k}\geq\epsilon)
\end{equation}

where $\tau_{y,k}=leb\left\{t \in [0,T_{\partial D(0,K\log k)}] \text{ s.t. } \mathcal{D}(\omega(t)) \in D(y,e^{-\lambda' k})\right\}$.
Let us prove inequality \eqref{equ4}. For this, we define:
%as in the proof of theorem 1, we have to deal with the following problem: $c_k(\tilde{\omega})$ and $y_k(\tilde{\omega})$ are not independent. To solve this problem, we remark that these two variables are independent conditionally to $\gamma_{N_k}(\tilde{\omega})=\gamma$. So defining:
$$U_{k,y}=leb\left\{t \in [S_{N_k},S_{N_{k+1}}] \text{ s.t. } \mathcal{D}(c_k(\tilde{\omega})(t)) \in D(y,e^{-\lambda' k})\right\}$$
$$V_{k,y}=leb\left\{t \in [S_{N_0},S_{N_1}] \text{ s.t. } \mathcal{D}(c_0(\tilde{\omega})(t)) \in D(y,e^{-\lambda' k})\right\}$$
As explained in the proof of theorem \ref{theoprincipal2}, $c_k$ and $y_k$ are independent. So, we have:
\begin{align*}
\tilde{\mathbb{P}}(T_k\geq\epsilon)
&\leq\sup_{y \in \mathbb{C}\mathbb{P}^1}\tilde{\mathbb{P}}(U_{k,y}\geq\epsilon)\\
&=\sup_{y \in \mathbb{C}\mathbb{P}^1}\tilde{\mathbb{P}}(V_{k,y}\geq\epsilon)\\
&=\sup_{y \in \mathbb{C}\mathbb{P}^1}\left(\tilde{\mathbb{P}}(\{V_{k,y}\geq\epsilon\}\cap A_k)+\tilde{\mathbb{P}}(\{V_{k,y}\geq\epsilon\}\cap A_k^c)\right)\\
&\leq \tilde{\mathbb{P}}(A_k) + \sup_{y \in \mathbb{C}\mathbb{P}^1} \mathbb{P}_0(\tau_{y,k}\geq\epsilon).
\end{align*}
The last inequality is due to the fact that for all $y\in\mathbb{C}\mathbb{P}^1$, we have: $\{V_{k,y}\geq\epsilon\}\cap A_k^c\subset\{\tau_{y,k}\geq\epsilon\}\times[0;1]^{\mathbb{N}}$ which implies that $\tilde{\mathbb{P}}(\{V_{k,y}\geq\epsilon\}\cap A_k^c)\leq \mathbb{P}_0(\tau_{y,k}\geq\epsilon)$.

Now, we are going to bound the two terms of the previous line. For the term $\tilde{\mathbb{P}}(A_k)$, we have the following proposition:
\begin{pro}
\label{pro2}
There exists $K$ such that for $k$ big enough, $\tilde{\mathbb{P}}(A_k)\leq\frac{1}{k^2}$.
\end{pro}

\begin{proof}

In \cite[Proposition 2.15]{DD2}, the authors prove that there is $\alpha>0$ such that:
$$\mathbb{E}[e^{\alpha S_{N_1}}]=M<\infty$$
Using Markov inequality, one deduces:
$$\tilde{\mathbb{P}}[S_{N_1}\geq t]=\tilde{\mathbb{P}}[e^{\alpha S_{N_1}}\geq e^{\alpha t}]\leq e^{-\alpha t}\mathbb{E}[e^{\alpha S_{N_1}}]=M e^{-\alpha t}$$
 
If $\omega$ is a Brownian path, denote:
\[\xi_t(\omega)=\sup_{0\leq u\leq t}d(\omega(0),\omega(t))\]
Let $C_1>0$ satisfying $\alpha C_1>2$.
We have:
\begin{align*}
\tilde{\mathbb{P}}(A_k)
&\leq \tilde{\mathbb{P}}(A_k\cap \{S_{N_1}\geq C_1 \log(k)\})+\tilde{\mathbb{P}}(A_k\cap \{S_{N_1}\leq C_1 \log(k)\})\\
&\leq \tilde{\mathbb{P}}(S_{N_1}\geq C_1 \log(k))+\mathbb{P}_0(\xi_{C_1 \log(k)}\geq K\log(k))
\end{align*}
The first term of the right hand side satisfies for $k$ big enough:
$$\tilde{\mathbb{P}}(S_{N_1}\geq C_1 \log(k))\leq M e^{-\alpha C_1 \log k}\leq \frac{1}{2k^2}$$

In order to bound the second term, we will use the following estimate (see \cite[paragraph 6]{P} for a proof): there is $c>0$ such that for all $y\in \mathbb{D}$ and for all  $r\geq 2$, we have $\mathbb{P}_y(\xi_1\geq r)\leq e^{-c r^2}$. Hence:

\begin{align*}
\mathbb{E}_y[e^{\xi_1}]
&=1+\int_{u>0}e^u \mathbb{P}_y(\xi_1\geq u)du\\
&\leq 1+\int_{u>0}e^{u-cu^2}du
\end{align*}

The last integral converges. Let $a_4$ be the constant satisfying $e^{a_4}=1+\displaystyle{\int_{u>0}}e^{u-cu^2}du$. Denote $\lfloor t \rfloor$ the integral part of $t$. Using successively the Markov inequality and the strong Markov property of Brownian motion, one gets:

\begin{align*}
\mathbb{P}_0(\xi_t\geq r)
&\leq e^{-r} \mathbb{E}[e^{\xi_t}]\\
&\leq e^{-r} \mathbb{E}\left[\exp(\sum_{k=0}^{\lfloor t\rfloor-1}\sup_{k\leq s\leq k+1}d(\omega(k),\omega(s)))\right]\\
&\leq e^{-r}.(sup_{y \in \mathbb{D}} \mathbb{E}_y[e^{\xi_1}])^t
\end{align*}

 For $t=C_1\log(k)$ and $r=K \log(k)$, one gets:
\[\mathbb{P}_0\left(\xi_{C_1 \log(k)}\geq K \log(k)\right)\leq k^{- K}\cdot k^{a_4C_1}\]

Consequently:

\[\tilde{\mathbb{P}}(A_k)\leq \frac{1}{2k^2}+k^{- K+a_4 C_1}\]
Choose $K$ big enough so that $- K+a_4 C_1<-2$. We get that for $k$ big enough: $\tilde{\mathbb{P}}(A_k)\leq \frac{1}{k^2}$.

\end{proof}

%For the term $Card(B_n)$, we have the following lemma:
%\begin{lem}
%\label{lemme2}
%$Card(B_n)\leq n^{\alpha}$
%\end{lem}

Up to now, we fix $K$ satisfying the previous proposition. In order to bound the second term $\sup_{y \in \mathbb{C}\mathbb{P}^1}\mathbb{P}_0(\tau_{y,k}\geq\epsilon)$ of \eqref{equ4}, we will need the following:
\begin{pro}
\label{pro3}
There exist two positive constants $\alpha$ and $\beta$ such that for all $y\in\mathbb{C}\mathbb{P}^1$ and for $k$ big enough, the intersection of $\mathcal{D}^{-1}(D(y,e^{\lambda'k}))$ with $D(0,K\log k)$ is included in an union of at most $k^{\alpha}$ discs with radius $e^{-\beta k}$.
\end{pro}

\begin{proof}
Let us fix $y\in\mathbb{C}\mathbb{P}^1$. Let $F$ be the Dirichlet fundamental domain associated to the base point $0\in \mathbb{D}$:
$$F=\left\{x\in\mathbb{D}\text{ s.t. }\forall \gamma \in \pi_1(\Sigma),\, d(0,x)\leq d(\gamma\cdot 0,x)\right\}$$
Let $D=F\cap D(0,K\log k)$. Note first that:
\begin{equation}
\label{eq5}
D(0,K\log k)\subset \underset{ d(0,\gamma.0)\leq 2K\log k}\bigcup \gamma\cdot D,
\end{equation}
To see this, take $z\in D(0,K\log k)$. There exists $\gamma\in\pi_1(\Sigma)$ such that $z\in\gamma\cdot F$. We have $d(0,\gamma^{-1}\cdot z)=d(\gamma\cdot 0,z)\leq d(0,z)\leq K\log k$. So $z\in\gamma \cdot D$. Moreover, $d(0,\gamma\cdot 0)\leq d(0,z)+d(z,\gamma\cdot 0)\leq 2\cdot d(0,z)\leq 2K\log k$.

Secondly, the $\rho$-equivariance of $\mathcal{D}$ gives for every $\gamma\in\pi_1(\Sigma)$:
$$\mathcal{D}^{-1}\left(D(y,e^{-\lambda'k})\right)\cap\gamma D=\gamma\cdot \left(\mathcal{D}^{-1}(\rho(\gamma^{-1})D(y,e^{-\lambda'k}))\cap D\right)$$
A direct calculation gives:
\begin{equation}
\label{eq5a}
||\rho(\gamma)||^2=\underset{z\in\mathbb{C}\mathbb{P}^1}\sup|\rho(\gamma^{-1})'(z)|.
\end{equation}
Indeed, if $\rho(\gamma^{-1})=\begin{pmatrix}
a&b \\
c&d
\end{pmatrix}$, then we have:
\begin{align*}
||\rho(\gamma)||^2&=\sup_{V\in\mathbb{C}^2-\{0;0\}}\frac{||\rho(\gamma)\cdot V||^2}{||V||^2}\\
                  &=\sup_{W\in\mathbb{C}^2-\{0;0\}}\frac{||W||^2}{||\rho(\gamma^{-1})\cdot W||^2}\\
								  &=\sup_{(z,w)\in\mathbb{C}^2-\{0;0\}}\frac{|z|^2+|w|^2}{|az+bw|^2+|cz+dw|^2}\\
									&=\sup_{z\in\mathbb{C}}\frac{|z|^2+1}{|az+b|^2+|cz+d|^2}\\
									&=\sup_{z\in\mathbb{C}}\frac{1}{|cz+d|^2}\cdot\frac{1+|z|^2}{1+|\rho(\gamma^{-1})(z)|^2}\\
									&=\sup_{z\in\mathbb{C}\mathbb{P}^1}|\rho(\gamma^{-1})'(z)|.
\end{align*}
Moreover, the monodromy being parabolic, we have already seen at the beginning of the proof of theorem \ref{theoprincipal}, that there exists a constant $a$ such that:
\begin{equation}
\label{eq5b}
\log||\rho(\gamma)||\leq a\cdot d(0,\gamma\cdot 0).
\end{equation}
From equations \eqref{eq5a} and \eqref{eq5b}, we deduce the following: if $\gamma\in\pi_1(\Sigma)$ is such that $d(0,\gamma\cdot 0)\leq 2K\log k$, then,
$$\underset{z\in\mathbb{C}\mathbb{P}^1}\sup|\rho(\gamma^{-1})'(z)|\leq k^{4aK},$$
which implies
$$\rho(\gamma^{-1})D\left(y,e^{-\lambda'k}\right)\subset D\left(\rho(\gamma^{-1})y, k^{4aK}\cdot e^{-\lambda'k}\right).$$
If $\alpha_1$ is a constant such that $0<\alpha_1<\lambda'$, we have for $k$ big enough: $ k^{4aK}\cdot e^{-\lambda'k}\leq e^{-\alpha_1 k}$. This implies that, for $k$ big enough and $\gamma$ satisfying $d(0,\gamma\cdot 0)\leq 2K\log k$:
\begin{equation}
\label{equ6}
\mathcal{D}^{-1}\left(D(y,e^{-\lambda'k})\right)\cap\gamma D\subset \gamma\cdot \left(\mathcal{D}^{-1}(D(\tilde{y},e^{-\alpha_1k}))\cap D\right)
\end{equation}
with $\tilde{y}=\rho(\gamma^{-1})y$. To conclude, we will need the following lemma: 

\begin{lem}
\label{lem}
There exist constants $N\in\mathbb{N}$ and $\beta>0$ such that for $k$ big enough and for every $\tilde{y}\in\mathbb{C}\mathbb{P}^1$, the set $\mathcal{D}^{-1}(D(\tilde{y},e^{-\alpha_1k}))\cap D$ is included in an union of at most $N$ discs with radius at most $e^{-\beta k}$.
\end{lem}

Before proving the lemma, let us finish the proof of proposition \ref{pro3}. Using \eqref{equ6} and the previous lemma, we get that for all $\gamma$ satisfying $d(0,\gamma\cdot 0)\leq 2K\log k$, the set $\mathcal{D}^{-1}\left(D(y,e^{-\lambda'k})\right)\cap\gamma D$ is included in an union of at most $N$ discs with radius at most $e^{-\beta k}$. Now, noting that there is $\alpha>0$ such that $Card\left\{\gamma\in\pi_1(\Sigma)\text{ s.t. } d(0,\gamma.0)\leq 2K\log k\right\}\leq k^{\alpha}$ and using equation \eqref{eq5}, we get the desired result.

\end{proof}

\paragraph{Proof of lemma \ref{lem}.}
Recall that the projective structure being parabolic, for any puncture $p$ in $\Sigma$, there is a neighborhood $V$ of $p$ satisfying the following: if $\mathcal{H}$ is the connected component of $proj^{-1}(V)$ which meets the fundamental domain $F$, then there is a biholomorphism bilipschitz $\tilde{\phi}:\mathbb{H}_{\geq 1}\to \mathcal{H}$ such that some developing map satisfies $\mathcal{D}\circ\tilde{\phi}(\tau)=\tau$. Denote $\mathcal{H}_1,\cdots,\mathcal{H}_r$ the set of all such component for each puncture.
Recall that $D=F\cap D(0,K\log k)$ and denote $F_0=D\cap(\cup_i\mathcal{H}_i)^c$. We are going to analyse the intersection of $\mathcal{D}^{-1}\left(D(\tilde{y},e^{-\lambda'n})\right)$ with the compact part $F_0$ and with $D-F_0$ separately. 

Firstly, consider the compact part $F_0$. Let us start with a heuristic argument: $\mathcal{D}'$ has a finite number of zeros $a_i$ in $F_0$. Let $V_i$ be a small neighborhood of $a_i$. $|\mathcal{D}'|$ is bounded away from $0$ on $F_0-\cup V_i$. So, if $\tilde{y}\in\mathbb{C}\mathbb{P}^1$, and $\alpha$ is small, $\mathcal{D}^{-1}\left(D(\tilde{y},\alpha)\right)(\cap F_0-\cup V_i)$ is a finite union of discs with radius of the order of $\alpha$. For each $i$, in local coordinates (for $V_i$ and $\mathcal{D}(V_i)$) the map $\mathcal{D}$ writes: $\mathcal{D}(z)=z^{n_i}$. This implies that $\mathcal{D}^{-1}\left(D(\tilde{y},\alpha)\right)\cap V_i$ is the union of at most $n_i$ discs with radius at most $\alpha^{\frac{1}{n_i}}$. Now, we will give a rigorous proof:  Denote by $N_{\epsilon}(F_0)=\{\tau\in\mathbb{D}\text{ s.t. } d(\tau,F_0)\leq\epsilon\}$ the $\epsilon$-neighborhood of $F_0$. As $\mathcal{D}$ is a non constant holomorphic map, there is a constant $N$ such that any $\tilde{y}$ has at most $N$ preimages by $\mathcal{D}$ in $N_{\epsilon}(F_0)$. Moreover, in \cite[Lemma 5.1]{AH}, we proved the following fact: $\exists C_0>0$ such that for any $\tilde{y}\in\mathcal{D}(N_{\epsilon}(F_0))$ and any $z\in F_0$:
$$d(\mathcal{D}(z),\tilde{y})\geq C_0 \prod_{\mathcal{D}(w)=\tilde{y},w\in N_{\epsilon}(F_0)}d(z,w)$$
Let $\tilde{y}\in\mathbb{C}\mathbb{P}^1$. If $\tilde{y}\notin \mathcal{D}(N_{\epsilon}(F_0))$, then for $k$ big enough $\mathcal{D}^{-1}\left(D(\tilde{y},e^{-\lambda'k})\right)\cap F_0=\varnothing$. Otherwise, $\tilde{y}\in \mathcal{D}(N_{\epsilon}(F_0))$. Then if $N(\tilde{y})$ denote the number of preimages of $\tilde{y}$ in $N_{\epsilon}(F_0)$, and if one takes $z\in\mathcal{D}^{-1}\left(D(\tilde{y},e^{-\lambda'k})\right)\cap F_0$, one gets:
\begin{align}
e^{-\lambda'k}&\geq d(\mathcal{D}(z),\tilde{y})\\
              &\geq C_0 \prod_{\mathcal{D}(w)=\tilde{y},w\in F_0+\epsilon}d(z,w)\\
							&\geq C_0 \left(\inf_{\mathcal{D}(w)=\tilde{y},w\in N_{\epsilon}(F_0)}d(z,w)\right)^{N(\tilde{y})}
\end{align}

This implies:
$$\inf_{\mathcal{D}(w)=\tilde{y},w\in N_{\epsilon}(F_0)}d(z,w)\leq \left(\frac{e^{-\lambda'k}}{C_0}\right)^{\frac{1}{N}}$$

As there exists $\beta$ such that for $k$ big enough $\left(\frac{e^{-\lambda'k}}{C_0}\right)^{\frac{1}{N}}\leq e^{-\beta k}$, we get that $z\in D(w,e^{-\beta k})$, for $w$ a preimage of $\tilde{y}$ by $\mathcal{D}$ in $F_0+\epsilon$.

Now, we are going to analyse the intersection of $\mathcal{D}^{-1}\left(D(\tilde{y},e^{-\lambda'k})\right)$ with each portion of horodisc $D\cap\mathcal{H}_i$. Recall that for each $i$, there is a biholomorphism bilipschitz $\tilde{\phi}_i:\{Im(z)\geq 1\}\to\mathcal{H}_i$ such that some developing map satisfies $\mathcal{D}\circ\tilde{\phi}_i(z)=z$ (see remark \ref{rem1}). $\tilde{\phi}_i$ being bilipschitz, it preserves the lengths modulo multiplications by constants. Hence, we can assume that $\mathcal{H}_i=\{Im(z)\geq 1\}$, the developing map is the inclusion $i:\{Im(z)\geq 1\}\rightarrow \mathbb{C}\mathbb{P}^1$ and $D\cap\mathcal{H}_i=D_{hyp}(i,K\log k)\cap [-\frac{1}{2},\frac{1}{2}]\times[1,+\infty[$. To evaluate the size of the preimage of the intersection of a disc with spherical radius $e^{-\alpha_1 k}$ with $D\cap\mathcal{H}_i$, we just have to compare the spherical metric $ds_{sph}$ and the hyperbolic one $ds_{hyp}$ inside $D_{hyp}(i,K\log k)\cap [-\frac{1}{2},\frac{1}{2}]\times[1,+\infty[$. We have:
$$ds_{hyp}=\frac{1+x^2+y^2}{y}.ds_{sph}$$
Furthermore, there is $\alpha>0$ such that $D_{hyp}(i,K\log k)\cap [-\frac{1}{2},\frac{1}{2}]\times[1,+\infty[\subset [-\frac{1}{2},\frac{1}{2}]\times [1;k^{\alpha}]$, so that $ds_{hyp}\leq(\frac{5}{4}+k^{2\alpha})\cdot ds_{sph}$.
This implies that a disc with spherical radius $e^{-\alpha_1 k}$ is included in a disc with hyperbolic radius $e^{-\alpha_1 k}\cdot (\frac{5}{4}+k^{2\alpha})$. There is $\beta$ such that for $k$ big enough $e^{-\alpha_1 k}\cdot (\frac{5}{4}+k^{2\alpha})\leq e^{-\beta k}$, this implies the desired result.

\paragraph{End of the proof of proposition \ref{pro1}}
Using the previous proposition, we are going to give a bound for the last term of equation \eqref{equ4}, namely we are going to prove that for all $\epsilon>0$, for k big enough:
\begin{equation}
\label{equ7}
\sup_{y \in \mathbb{C}\mathbb{P}^1} \mathbb{P}_0(\tau_{y,k}\geq\epsilon)\leq \frac{1}{k^2}.
\end{equation}
Then a combination of inequalities \eqref{equ4}, \eqref{equ7} and of proposition \ref{pro2} implies that $\tilde{\mathbb{P}}(T_k\geq\epsilon)\leq \frac{2}{k^2}$, which ends the proof of proposition \ref{pro1} and that of the theorem.

The proof of equation \eqref{equ7} goes as follows: fix $y\in\mathbb{P}_1$ and $\epsilon>0$. We recall that:
$$ \tau_{y,k}=leb\left\{t \in [0,T_{\partial D(0,K\log k)}] \text{ s.t. } \mathcal{D}(\omega(t)) \in D(y,e^{-\lambda' k})\right\}$$
and that according to proposition \ref{pro3}, we have:
$$\mathcal{D}^{-1}(D(y,e^{\lambda'k}))\bigcap D(0,K\log k)\subset\bigcup_{i\in I} D_i$$
where $D_i$ are hyperbolic discs with radius $e^{-\beta k}$ and $Card(I)\leq k^{\alpha}$.
We have:
\begin{align*}
 \mathbb{P}_0(\tau_{y,k}\geq\epsilon)
                            &\leq\frac{1}{\epsilon}\cdot \mathbb{E}_0[\tau_{y,k}]\\
														&\leq\frac{1}{\epsilon}\cdot \sum_{i\in I}\mathbb{E}_0[\tau_{D_i}],
\end{align*}
where $\tau_{D_i}=leb\left\{t \in [0,\infty[ \text{ s.t. } \omega(t) \in D_i\right\}$.
Let $D_i$ be such an hyperbolic disc with hyperbolic radius $e^{-\beta k}$.
$$\mathbb{E}_0[\tau_{D_i}]=\int_{D_i}G_{\mathbb{D}}(0,z)dhyp(z)$$
where $G_{\mathbb{D}}(0,z)=-\frac{1}{\pi}\log|z|$ is the Green function.
In order to give an upper bound of this integral, we distinguish two cases:
\begin{itemize}
\item Either $D_i\subset D_{eucl}(0,\frac{1}{4})^c$. In this case $-\log |z|\leq \log 4$ for every $z\in D_i$. Hence $\displaystyle{\int_{D_i}}G_{\mathbb{D}}(0,z)dhyp(z)\leq cst\cdot vol_{hyp}(D_i)\underset{k\to 0}\sim cst\cdot e^{-2\beta k}$.
\item Or $D\subset D_{eucl}(0,\frac{1}{2})$. In this case, for every $z\in D_i$, we have $dhyp(z)=\frac{|dz|^2}{(1-|z|^2)^2}\leq \frac{16}{9}.|dz|^2$. In polar coordinates, $z=re^{i\theta}$ and $|dz|^2=r.dr.d\theta$. Hence:
$$\int_{D_i}G_{\mathbb{D}}(0,z)dhyp(z)\leq cst\cdot \int_{D_i}-\log (r) rdrd\theta.$$
As $-r\log r\leq e^{-1}$ on $[0,1]$, we get that:
\begin{align*}
\int_{D_i}G_{\mathbb{D}}(0,z)dhyp(z)&\leq cst\cdot \int_{D_i}drd\theta\\
                                    &\leq cst\cdot e^{-\beta k}
\end{align*}
The last line is due to the fact that the hyperbolic disc $D_i$ with radius $e^{-\beta k}$ is also an euclidean disc with radius less than $e^{-\beta k}$.
\end{itemize}
So, for $k$ big enough, any hyperbolic disc $D_i$ with radius $e^{-\beta k}$, satisfies $\mathbb{E}_0[\tau_{D_i}]\leq cst\cdot e^{-\beta k}$. So we have:
$$\frac{1}{\epsilon}. \sum_{i\in I}\mathbb{E}_0[\tau_{D_i}]\leq cst\cdot Card(I).e^{-\beta k}\leq cst\cdot k^{\alpha}\cdot e^{-\beta k}\leq \frac{1}{k^2}$$
for $k$ big enough.

Nicolas Hussenot Desenonges\newline
  Instituto de Matem\'{a}tica, Universidade Federal do Rio de Janeiro, \newline
  Ilha do Fundao, 68530, CEP 21941-970, Rio de Janeiro, RJ, Brasil \newline
  \textit {e-mail}: nicolashussenot@hotmail.fr

\end{document}